\documentclass[10pt, reqno]{amsart}
\usepackage[inner=3cm,outer=3cm,bottom=4cm,top=4cm]{geometry}
 
\usepackage[square, numbers, comma]{natbib} 
\usepackage{caption}
\usepackage[english]{babel}
\usepackage{microtype}
\usepackage{amsfonts}
\usepackage[utf8]{inputenc}	
\usepackage[T1]{fontenc}   
\usepackage{yfonts}
\usepackage{amsmath}
\usepackage{xfrac}
\usepackage{graphicx}
\usepackage[colorlinks=true, allcolors=black]{hyperref}
\usepackage{enumitem}
\usepackage{tikz}
\usepackage{amssymb}
\usepackage{amsthm}
\usepackage{color}
\usepackage{subcaption}
\usepackage{afterpage}
\usepackage{mathtools}
\usepackage{pdfpages}
\usepackage{bm}
\usepackage{bbm}
\usepackage{appendix}
\usepackage{cleveref}
\usepackage{tikz}
\usepackage{comment}
\usepackage{float}
\usepackage{empheq}

\crefformat{equation}{(#2#1#3)}
\crefmultiformat{equation}{(#2#1#3)}%
{ and~(#2#1#3)}{, (#2#1#3)}{ and~(#2#1#3)}
\crefrangeformat{equation}{(#3#1#4) to~(#5#2#6)}

\newtheorem{theorem}{Theorem}[section]
\newtheorem{lemma}[theorem]{Lemma}
\newtheorem{proposition}[theorem]{Proposition}

\theoremstyle{definition}

\newtheorem{remark}[theorem]{Remark}

\newcommand{\spann}{\mathrm{span}}

\newcommand{\R}{\mathbb{R}}
\newcommand{\N}{\mathbb{N}}
\newcommand{\Z}{\mathbb{Z}}
\newcommand{\T}{\mathbb{T}}
\newcommand{\C}{\mathbb{C}}

\newcommand{\id}{\mathrm{id}}

\newcommand{\Log}{\mathrm{Log}}
\renewcommand{\phi}{\varphi}

\begin{document}

\title[The zero-dispersion limit for the Benjamin--Ono equation on the circle]{The zero-dispersion limit for the Benjamin--Ono equation on the circle}
%\date{January 2025}

\author[O.~M{\ae}hlen]{Ola M{\ae}hlen}
\address[O.~M{\ae}hlen]{Mathematical Institute of Orsay\\
Paris-Saclay University\\
91400 Orsay, France}
\email[]{ola.maehlen\@@{}universite-paris-saclay.fr}



\begin{abstract}
\noindent Using the explicit formula of P.~Gérard, we characterize the zero-dispersion limit for solutions of the Benjamin–Ono equation on the circle $\T\coloneqq \R/2\pi\Z$ with bounded initial data $u_0\in L^\infty(\T,\R)$. The result generalizes the work of L.~Gassot, who focused on periodic bell-shaped data, and complements the work of Gérard and X.~Chen who identified the zero-dispersion limit on the line with $u_0\in L^2\cap L^\infty(\R)$. Here, as well as in the mentioned cases, the characterization agrees with the one first obtained by Miller--Xu for bell-shaped data on the line: The zero-dispersion limit is given as an alternating sum of the branches of the multivalued solution of Burgers' equation. From this characterization, we compute regularity properties of the zero-dispersion limit, including maximum principles and an Oleinik estimate.
\end{abstract}
\maketitle
\section{Introduction}
\noindent The Benjamin–Ono equation takes the form
\begin{equation}\label{eq: BO}\tag{BO}
    \partial_t u + \partial_x(u^2-|D|u) =0,
\end{equation}
where $u\colon \R_t\times X_x\to \R$ with $X\in \{\R,\T\}$, and $|D|$ is the nonlocal Fourier multiplier $\mathcal{F}(|D|u)(t,\xi)=|\xi|\mathcal{F}(u)(t,\xi)$.\footnote{In the literature, this operator is occasionally written in terms of the Hilbert transform $|D|=\mathcal{H}\circ\partial_x=\partial_x\circ \mathcal{H}$.} It originates from the works of Benjamin \cite{Benjamin_1967}, Davis--Acrivos \cite{Davis_Acrivos_1967}, and Ono \cite{Ono}, as a model for deep internal waves in stratified fluids, an application that was recently rigorously justified \cite{Martin}. Like the KdV equation, \eqref{eq: BO} is completely integrable. And while this has been known for some time \cite{Nakamura,BockKruskal}, it was only recently discovered that the equation admits Birkhoff coordinates \cite{Birkhoff} and, more remarkably, that its solutions can be described by an explicit formula \cite{ExplicitFormulaLine}. These discoveries have led to a flurry of new results of which we mention a few: It is now known that \eqref{eq: BO} is well posed in the Sobolev space $H^{s}(X,\R)$ for every choice $s>-\frac{1}{2}$, and that these results are sharp \cite{SharpWellPosednessCircle,SharpWellPosedness}.\footnote{See the introduction in \cite{SharpWellPosedness} for an up to date summary of the well-posedness theory for Benjamin--Ono.} Furthermore, one can characterize the long-term dynamics of these solutions. On the circle $\T$, the solutions are Bohr almost periodic in time \cite{SharpWellPosednessCircle,Periodicity}, while on the line $\R$ the soliton resolution conjecture has been resolved \cite{solitonResolutionBO,SolitonResolutionConjecture} for all $H^1(\R)$ data with sufficient decay. Another application of these novel tools for \eqref{eq: BO} is in determining its zero-dispersion limit, which is the subject of interest here.

The problem goes as follows. Fix some initial datum $u_0$ and consider the solutions $(u^\varepsilon)_{\varepsilon>0}$ of the small-dispersion Benjamin–Ono equation
\begin{equation}\label{eq: BOepsilon}\tag{BO-$\varepsilon$}
        \partial_tu^\varepsilon + \partial_x\big((u^\varepsilon)^2-\varepsilon|D|u^\varepsilon\big) =0, \qquad u^\varepsilon(0,x)=u_0(x).
\end{equation}
When $\varepsilon=0$, we get the (globally ill-posed) Burgers equation where shocks/infinite slopes form in finite time.  But as soon as $\varepsilon>0$, well-posedness is ensured and shocks are replaced by rapid oscillations, so-called \textit{dispersive} shocks. As $\varepsilon\to 0$, these dispersive shocks become increasingly oscillatory inhibiting the existence of a strong limit, but leaving room for a \textit{weak} limit. This limit, if it exists, is called the \textit{zero-dispersion limit}, and the relevant problem is to both prove its existence and to characterize it. This is not a problem specific to \eqref{eq: BO}, and it was first considered for KdV. Since the origin of the subject is not so well-known, we give a brief account of it; for further details see \cite{LaxSummary,TrioSamarbeid}.\\

\subsection{Some history and related works} The story of zero-dispersion limits began with the discovery of numerical dispersion. In 1944, John von Neumann used a central difference scheme to model compressible fluid flows and observed that, where shocks should form, large-amplitude oscillations at mesh scale arose in stead \cite{vonNeumann}. He conjectured that these oscillations would converge weakly to the correct solution as the resolution of the mesh increased. Unconvinced by this claim, Peter Lax caught an interest in the phenomenon and, with collaborators, wrote a series of papers on the topic \cite{LaxConjecture,LaxLevermore,LaxGoodman,LaxChinese,LaxSummary}.\footnote{In \cite{LaxChinese} Lax concludes that von Neumann was wrong in his assertion; the scheme admits a weak limit, but the limit fails to satisfy the underlying equations.} Lax viewed numerical dispersion as a discrete analogue \cite{LaxSummary} of the oscillations featured in weakly dispersive PDEs (of the form analogous to \eqref{eq: BOepsilon}). Aiming to better understand numerical dispersion, this viewpoint served as motivation behind his three prominent papers with Levermore \cite{LaxLevermore} where they compute the zero-dispersion limit of the KdV equation. These papers have since sparked significant research into small-dispersion effects. We present a few of the results in this direction:

First, for KdV, we mention that Venakides \cite{VenakidesGeneral,VenakidesPeriod,VenikadesOscillations} built upon the work of Lax and Levermore by generalizing the admissible initial data (allowing, in particular, for periodic data) and by identifying leading-order asymptotics for the dispersive shocks; see \cite{DeiftVenakidesZhou} for a rigorous derivation of the latter result. Next, after KdV, the zero-dispersion limit was analyzed for NLS and mKdV, the second and third members, respectively, of the Zakharov--Shabat hierarchy. All members of the defocusing hierarchy were treated in \cite{DefocusingNLS}, and the \textit{odd} members of the focusing hierarchy were treated in \cite{FocusingNLS}. The latter work did not include focusing NLS which was instead the subject of \cite{NLSOscillations,MillerNLS} where its zero-dispersion limit, and leading-order asymptotics of the oscillations, were computed. We also mention that analogous results were derived for the Toda lattice \cite{TodalatticZeroDispersionLimit,TodaOscillations} around the same time period. Despite the complete integrability of the systems mentioned so far, dispersive shocks are a generic phenomenon \cite[Sec. 4]{SurveyPaper}. In particular, Dubrovin's universality conjecture \cite{DubrovinUniversality2} (which is supported by numerical evidence \cite{NumericalEvidence}, and admits various generalizations \cite{DubrovinKleinConjecture,DubrovinKleinNewConjecture, Masoero}\footnote{See \cite{Masoero,millerUniversality} for the corresponding conjecture for \eqref{eq: BO}.}) postulates precise asymptotic behavior of dispersive shocks for a large class of Hamiltonian equations --- most of which are not integrable. Thus, while not the mechanism behind these turbulent phenomena, integrability offers the means to study them. 

We return to the Benjamin–Ono equation. Many authors have contributed to the study of its zero-dispersion limit; we choose here to focus on a  selection of works beginning with that of Miller and Xu \cite{Miller1} (see their introduction for a broader overview). By an adaption of the Lax–Levermore theory for KdV, Miller--Xu were able to determine the zero-dispersion limit of \eqref{eq: BO} for bell-shaped initial data on the line. They found that this limit admitted a surprisingly simple description in terms of the multivalued solution of Burgers' equation. To state their result, we first recall that the Burgers equation, here written 
\begin{equation}\label{eq: burgersEquation}\tag{Bu}
    \partial_tu + \partial_x(u^2)=0,\quad u(0,x)=u_0(x),
\end{equation}
can be solved by the method of characteristics, yielding the following implicit solution formula
\begin{align}\label{eq: methodOfCharacteristics}
    u_B(t,x) = u_0(x-2tu_B(t,x)).
\end{align}
It is well-known that the resulting solution $u_B$ will not be uniquely determined by \eqref{eq: methodOfCharacteristics} for $t$ greater than the breaking time $T\coloneqq-(\inf_{x}2u_0'(x))^{-1}$ at which point the characteristics intersect (a shock has formed). But we embrace this defect and let $u_B$ be the multivalued function of values satisfying \eqref{eq: methodOfCharacteristics}. The aforementioned result of Miller--Xu can then be summarized as follows. For appropriately bell-shaped\footnote{See Definition 3.1.~in \cite{Miller1} for the precise assumptions.} initial data $u_0$ on the line, they demonstrate that the solutions $u^\varepsilon$ of \eqref{eq: BOepsilon} admit, for any $t>0$, the following weak limit in space as $\varepsilon\to 0$,
\begin{equation}\label{eq: alternatingSumFormula}
\mathop{\operatorname{w-lim}}_{\varepsilon\to 0}u_{\varepsilon}(t)=ZD[u_0](t),\text{ where }ZD[u_0](t,x)=\sum_{n=0}^{N}(-1)^{n}u_B^n(t,x),
\end{equation}
with $u_B(t,x)=\{u_B^0(t,x)<u_B^1(t,x)<\dots<u_B^N(t,x)\}$; the alternating sum formula in \eqref{eq: alternatingSumFormula} presupposes that $(t,x)$ avoids the null set of caustics arising from the characteristics\footnote{See Remark \ref{rem: failureAtCaustics} for a discussion on what goes wrong at the caustics}. Furthermore, they proved that $u^\varepsilon\to ZD[u_0]$ strongly for $t<T=-(\inf_{x}2u_0'(x))^{-1}$; an indication that the dispersive shocks arise precisely at the same time as classical shocks form. It is worth noting that the corresponding formula for the zero-dispersion limit of KdV, is much more implicit than \eqref{eq: alternatingSumFormula}.

While Miller and Xu posed strong restrictions on $u_0$ (necessitated by their method), they expected that their conclusion \eqref{eq: alternatingSumFormula} held for more general data. An alternate approach to the problem opened up when Gérard
discovered the explicit formula for \eqref{eq: BO} \cite{ExplicitFormulaLine}. In \cite{ZeroDispersionLimitGerard}, he extended the result \eqref{eq: alternatingSumFormula} to all initial data $u_0\in L^2\cap L^\infty(\R)$; shortly after, Chen \cite{ZeroDispersionLimitXiChen} extended the result to $L^2\cap L^\infty_{\mathrm{loc}}(\R)$ data of sublinear growth at infinity. In the spirit of Venakides for KdV, a recent paper \cite{OscillationsBO} also provides a leading-order description of the dispersive shocks of \eqref{eq: BO} for rational initial data on the line. 

However, corresponding results on the circle $\T$ (or equivalently, for periodic initial data), are almost nonexistent. The only proof of a zero-dispersion limit here is given by Gassot \cite{ZeroDispersionLouiseGassot}. Using the Birkhoff coordinates \cite{Birkhoff} of \eqref{eq: BO}, she computes zero-dispersion limit corresponding to bell-shaped initial data; the limit yet again agrees with \eqref{eq: alternatingSumFormula}. The restrictions on $u_0$ allowed Gassot to obtain asymptotic expressions, in the small-dispersion limit, for the eigenvalues of the corresponding Lax operator\footnote{It is for the same reason that Miller--Xu \cite{Miller1} require bell-shaped initial data. Such asymptotic expressions were first derived by Matsuno \cite{MATSUNO1,MATSUNO2}.} as needed to compute the Birkhoff coordinates. In contrast, the explicit formula for \eqref{eq: BO}, discovered after Gassot’s work, requires no knowledge of these eigenvalues and thus provides a more direct route to the zero-dispersion limit; this is the approach we take here.

\subsection{Strategy and main results}
We outline the paper as follows: In Section \ref{sec: preliminaries}, we scale the explicit formula for \eqref{eq: BO} to get an explicit formula for \eqref{eq: BOepsilon}. And after sending $\varepsilon\to0$, we obtain an expression for (and hence the existence of) the zero-dispersion limit $ZD[u_0]$; this is Proposition \ref{prop: existenceOfZDLimit}. However, it is not obvious that this expression coincides with the simpler `alternating sum formula' obtained by Miller--Xu and generalized to the $L^\infty$-setting by Proposition \ref{prop: generalizedASFormula}. Their equivalence is proved in Section \ref{sec: provingTheMainResults}, where the first step is to series expand the two expressions (Proposition \ref{prop: seriesExpansions}). Using Lemma \ref{lem: theRaney-typeLemma}, which has a combinatorial flavor, the two series are rearranged so to make their equivalence evident; this culminates in Theorem \ref{thm: characterizingTheZeroDispersionLimit}. With the newly obtained simpler expression for $ZD[u_0]$, additional regularity properties are derived as summed up in Theorem \ref{thm: propertiesOfZD}.

The approach here is analogous to the approach of Chen in \cite{ZeroDispersionLimitXiChen}, where the zero-dispersion limit of \eqref{eq: BO} on the line was characterized by a series expansion argument using the corresponding explicit formula. There too, a lemma of combinatorial nature played a part, namely the identity:
\begin{equation}\label{eq: theCombinatoricsResultOfXiChen}
    \int_{\R}T_{v}^k(1)\, dx = \frac{1}{k}\int_{\R}v^k\, dx,\qquad k\in \N,
\end{equation}
where $v\in L^1\cap L^\infty(\R,\R)$ and $T_v$ is a Toeplitz operator on the line (defined analogously to those defined on the circle in Section \ref{sec: hardySpacesAndToeplitzOperators}). This identity has an interesting indirect proof as explained in \cite[Section 3.2]{Miller1}: Each side of \eqref{eq: theCombinatoricsResultOfXiChen} is a zero-dispersion limit of the $k$'th conservation law for \eqref{eq: BO} written in a canonical form.\footnote{Identity \eqref{eq: theCombinatoricsResultOfXiChen} is also proved more directly in \cite{ZeroDispersionLimitXiChen} and \cite[Appendix]{Miller1}.} In this paper, the corresponding combinatorics result, Lemma \ref{lem: theRaney-typeLemma}, is proved directly; we have not found an analogous indirect proof.

Our main results are proved in Section \ref{sec: provingTheMainResults}, and read as follows:

\begin{theorem}[Characterizing the zero-dispersion limit]\label{thm: characterizingTheZeroDispersionLimit}
   Let $u_0\in L^\infty(\T,\R)$ and $u^\varepsilon\in C(\R, L^2(\T,\R))$ denote the unique solution\footnote{In the sense of Molinet \cite{MolinetPeriodic} or, equivalently, Gérard--Kappeler \cite{SharpWellPosednessCircle}.} of \eqref{eq: BOepsilon}. Then, as $\varepsilon\to 0$, $u^\varepsilon$ converges to a function $ZD[u_0]$ in $C(\R, L^2_{w}(\T,\R))$: For each $\varphi\in L^2(\T,\R)$ we have
   \begin{equation*}
     \lim_{\varepsilon\to 0}\langle u^\varepsilon(t),\varphi\rangle_{L^2(\T)}= \langle ZD[u_0](t),\varphi\rangle_{L^2(\T)},
   \end{equation*}
locally uniformly in $t\in\R$. The function $ZD[u_0](t)$ is, for a.e.~$x\in \T$, given by 
\begin{equation*}
    ZD[u_0](t,x)=\int_{\R}\chi(u_0(x-2ty),y)\, dy,\quad \chi(u,y)\coloneqq \mathbbm{1}_{u>y>0} - \mathbbm{1}_{u<y<0}.
\end{equation*}
As shown in Lemma \ref{lem: theObviousLemma} and illustrated in Figure \ref{fig: multivaluedBurgers}, this generalizes \eqref{eq: alternatingSumFormula}.
\end{theorem}
It is unclear if the previous theorem extends to all unbounded $u_0\in L^2(\T)$, though some cases can be included: It suffices that $u_0$ is bounded below, or above; see Remark \ref{rem: onTheLInfinityAssumption}. We also point out that the theorem holds true for $\varepsilon$-dependent initial data $u^\varepsilon_0$ that converges in the $L^2$-sense to $u_0\in L^\infty(\T)$ as $\varepsilon\to 0$; see Remark \ref{rem: initialDataCanDependOnEpsilon}.

The zero-dispersion limit map $(t,u_0)\mapsto ZD[u_0](t)$ is well-behaved and coincides with the `transport-collapse map' $(t,u_0)\mapsto T(t)u_0$ of Y. Brenier \cite{Brenier1981,Brenier1983,Brenier} used to approximate the entropy solutions of hyperbolic conservation laws. And in fact, several of the good properties of entropy solutions for Burgers' equation also hold for $t\mapsto ZD[u_0](t)$ as summed up in the next theorem. However, note that $ZD[u_0]$ will \textit{not} be a weak solution of Burgers' equation for sufficiently large $t$ if $u_0\neq\mathrm{const}$ (see \cite{ZeroDispersionLimitGerard}).
\begin{theorem}[Properties of $ZD{[\cdot]}$]\label{thm: propertiesOfZD}
    The mapping $ZD[\cdot]$ in Theorem \ref{thm: characterizingTheZeroDispersionLimit} extends uniquely to a continuous mapping $L^p(\T,\R)\to C(\R,L^p(\T,\R))$ for any $p\in [1,\infty)$, that is,
    \begin{equation*}
        ZD[u^\delta_0](t^\delta)\to ZD[u_0](t)\quad \text{in } L^p(\T,\R),
    \end{equation*}
    provided $(t^\delta, u_0^\delta)\to (t,u_0)$ in $\R\times L^p(\T,\R)$.
    Moreover, for $t\in\R$ we have\vspace{5pt}
    \begin{enumerate}
        \item Maximum principles:  \quad $\mathrm{ess}\inf_{x\in\T}u_0(x)\leq ZD[u_0](t) \leq \mathrm{ess}\sup_{x\in\T}u_0(x). $\vspace{8pt}
        \item Norm control:\quad
        $\|ZD[u_0](t)\|_{L^p}\leq \|u_0\|_{L^p},\quad  \text{for }p\in [1,\infty]. $\vspace{8pt}
        \item $L^1$-contractions:\quad $\|ZD[u_0](t)-ZD[v_0](t)\|_{L^1}\leq \|u_0-v_0\|_{L^1}.$\vspace{8pt}
        \item Oleinik estimate: $$\frac{2t}{x-y}\Big(ZD[u_0](t,x)-ZD[u_0](t,y)\Big)\leq 1,\quad \text{for }x\neq y.$$
        \item Strong convergence before breaking: If $u_0\in C^1(\T,\R)$, and $u^\varepsilon$ is the corresponding solution of \eqref{eq: BOepsilon}, then we have strong convergence
        \begin{equation*}
            \lim_{\varepsilon\to 0}\|u^\varepsilon(t)-ZD[u_0](t)\|_{L^2}=0,\quad \text{for }t\in [T_-,T_+],
        \end{equation*}
        where the breaking times are defined by
        \begin{equation}
            T_-\coloneqq -\frac{1}{\sup_{x\in\T}2u'_0(x)},\qquad   T_+\coloneqq -\frac{1}{\inf_{x\in\T}2u'_0(x)},
        \end{equation}
for nonconstant $u_0$, and otherwise by $T_\pm\coloneqq \pm\infty$ for when $u_0$ is constant (in which case $u^\varepsilon(t)=ZD[u_0](t)=u_0$ for all $t$ and $\varepsilon$).
        \item Entropy solution from Trotter's formula: If $u_0\in L^1(\T)$, then the unique entropy solution\footnote{In the sense of \cite{Kinetic}.} of Burgers' equation, $u_{\mathrm{ent}}\in C([0,\infty),L^1(\T,\R))$, can be obtained from infinitesimal compositions of $ZD$: 
        \begin{equation}\label{eq: trotterFormula}
            u_{\mathrm{ent}}(t)= \lim_{n\to\infty}\Big(ZD[\cdot](\tfrac{t}{n})\Big)^{\circ n}[u_0],\quad \text{in }L^1(\T),
        \end{equation}
    where $f^{\circ n}\coloneqq f\circ f\circ\dots\circ f$ ($n$ times).
    \end{enumerate}
\end{theorem}

\section{Preliminaries}\label{sec: preliminaries}
In this section we introduce the necessary framework and prove two smaller results, Propositions \ref{prop: existenceOfZDLimit} and \ref{prop: generalizedASFormula}, from which we shall later derive Theorem \ref{thm: characterizingTheZeroDispersionLimit}. 
\subsection{The Hardy spaces and Toeplitz operators}\label{sec: hardySpacesAndToeplitzOperators}
Let $C^\omega(\mathbb{D})$ denote the set of holomorphic functions on the open unit disc $\mathbb{D}=\{z\in\C\colon |z|<1\}$. We introduce the two Hardy spaces
\begin{align*}
    L^2_+(\T)=&\,\{f\in L^2(\T,\C)\colon \hat{f}(k)=0,\,\, \forall k<0\}, &
    \mathcal{H}^2(\mathbb{D})=&\,\{F\in C^\omega(\mathbb{D})\colon \textstyle\sum_{k=0}^\infty|F^{(k)}(0)|^2<\infty\}.
\end{align*}
We endow $L_+^2(\T)$ with its Hilbert space structure inherited from $L^2(\T,\C)$ where the latter is equipped with the inner product $\langle f,g\rangle \coloneqq\frac{1}{2\pi}\int_{\T}f\overline{g}\,dx$. We further introduce the associated orthogonal projection, the Szegő projector $\Pi\colon L^2(\T,\C)\to L^2_+(\T)$, which deletes any negative frequencies $\widehat{\Pi f}(k)=\mathbbm{1}_{k\geq 0}\hat{f}(k)$. The canonical bijection $L^2_+(\T)\to \mathcal{H}^2(\mathbb{D})$ is the mapping $f\mapsto F_f$ defined by
\begin{equation}\label{eq: theBijectionBetweenTheTwoHardySpaces}
F_f(z)=\sum_{k=0}^\infty \hat{f}(k)z^k,
\end{equation}
whose inverse is the $L^2_+(\T)$-limit $f(x)=\lim_{r\uparrow 1}F_f(re^{ix})$. We will make use of the following important observation: For any real-valued $g\in L^2(\T,\R)\subset L^2(\T,\C)$, it holds that
\begin{equation}\label{eq: u_0IsDerivableFromPiu_0}
    g=\Pi g + \overline{\Pi g} - \langle \Pi g,1\rangle,
\end{equation}
that is, $g$ is completely determined by $\Pi g\in L^2_+(\T)$ which again is completely determined by $F_{\Pi g}\in \mathcal{H}^2(\mathbb{D})$; for brevity, we will simply write $F_g$ in stead of $F_{\Pi g}$.

Next, with $b\in L^\infty(\T,\C)$ we define the Toeplitz operator $T_b\colon L^2_+(\T)\to L^2_+(\T)$ by
\begin{equation*}
    T_bf\coloneqq \Pi(bf),
\end{equation*}
for $f\in L^2_+(\T)$. Note that $T_b$ is linear and bounded, and that $T_b$ is self-adjoint whenever $b$ is real-valued: For $f,g\in L^2_+(\T)$ we have 
\begin{equation*}
    \langle \Pi(b f),g\rangle_{L^2_+(\T)}= \langle bf,g\rangle_{L^2(\T,\C)}=\langle f,bg\rangle_{L^2(\T,\C)}=\langle f,\Pi(bg)\rangle_{L^2_+(\T)}.
\end{equation*}
Two classical Toeplitz operators in our setting are the right- and left-shifts
\begin{equation*}
    S\coloneqq T_{q},\qquad\qquad S^*\coloneqq T_{\overline{q}},\qquad \text{where $q(x)\coloneqq  e^{ix}$.}
\end{equation*}
We stress that $S$ and $S^*$ do \textit{not} commute. Indeed, $S^*Sf=f$ while $SS^*f=f-\langle f,1\rangle$. A further illustration of this non-commutativity, is the following lemma that will be useful later.
\begin{lemma}[Compositions of Toeplitz-shifts]\label{lem: compositonsOfToeplitzShifts}
    For $y\in \Z$, define the $y$-shift by $S_y\coloneqq T_{q^y}$ where as before $q(x)=e^{ix}$. Then, $S_{y_1}S_{y_2}=S_{y_1+y_2}$ if, and only if, either $y_{1}\leq 0$ or $y_2\geq 0$. Moreover, for any finite sequence $y_1,\dots,y_J\in\Z$ we have
    \begin{equation}\label{eq: explicitCompositionOfToeplitzShifts}
        S_{y_1}S_{y_2}\cdots S_{y_J}(1)=\begin{cases}
            q^{y_1+\dots+y_J}, \text{ if } y_j+\dots + y_J\geq 0\text{ for all }j\in\{1,\dots,J\},\\
            0, \hspace{43pt}\text{else}.
        \end{cases}
    \end{equation}
\end{lemma}
\begin{proof}
By definition, we have $\Pi (q^k)= 0$ when $k<0$ and $\Pi (q^k)=q^k$ when $k\geq 0$; this is the only ingredient needed in the proof. To prove the first part of the lemma, we restrict our attention to $q^n$ for $n\in \N_0$, since $\spann\{q^{n}\colon n\in\N_0\}\subset L^2_+(\T)$ is dense and the shifts are bounded operators. A straight forward computation reveals that
\begin{equation}\label{eq: computingTwoShifts}
\begin{split}
    S_{y_1}S_{y_2}(q^n)=&\,\Pi(q^{y_1}\Pi(q^{y_2+n}))=\big(\mathbbm{1}_{y_1+y_2+n\geq 0}\big)\big(\mathbbm{1}_{y_2+n\geq 0}\big)q^{y_1+y_2+n}\\
    S_{y_1+y_2}(q^n)=&\, \Pi(q^{y_1+y_2+n})=\big(\mathbbm{1}_{y_1+y_2+n\geq 0}\big)q^{y_1+y_2+n}.
\end{split}
\end{equation}
If $y_1\leq 0$, then $y_1+y_2+n\geq 0$ only if $y_2+n\geq0$, which by \eqref{eq: computingTwoShifts} implies $S_{y_1}S_{y_2}(q^n)=S_{y_1+y_2}(q^n)$. If $y_2\geq 0$, then $y_2+n\geq 0$ holds trivially, and $S_{y_1}S_{y_2}(q^n)=S_{y_1+y_2}(q^n)$ again follows by \eqref{eq: computingTwoShifts}. Finally, if $y_1>0$ and $y_2<0$ then we set $n\coloneqq \max\{0,-(y_1+y_2)\}$ and observe that $y_1+y_2+n\geq 0$ while $y_2+n<0$, and so we get from \eqref{eq: computingTwoShifts} that $S_{y_1}S_{y_2}(q^n)=0\neq q^n=S_{y_1+y_2}(q^n)=q^n$. The second part of the lemma, equation \eqref{eq: explicitCompositionOfToeplitzShifts}, follows from a calculation similar to the first line in \eqref{eq: computingTwoShifts}. 
\end{proof}
\subsection{The explicit formula and its zero-dispersion limit}
Let $u_0\in L^\infty(\T,\R)$, and let $u\in C(\R, L^2(\T,\R))$ denote the corresponding solution of \eqref{eq: BO}. In light of \eqref{eq: u_0IsDerivableFromPiu_0} and the subsequent comment, $u(t)$ is completely determined by the function $F_{u(t)}\in \mathcal{H}^2(\mathbb{D})$ as defined in \eqref{eq: theBijectionBetweenTheTwoHardySpaces}. From \cite[Thm. 3]{ExplicitFormulaLine} (and Remark 1 therein), this last quantity is, for $t\in\R$ and $z\in\mathbb{D}$, given by
\begin{equation}\label{eq: explicitFormula}
    F_{u(t)}(z)=\big\langle\big(\mathrm{id}-ze^{it}e^{2it(D-T_{u_0})}S^*\big)^{-1}\Pi u_0,1\big\rangle,
\end{equation}
where $D\coloneqq -i\partial_x$ and $\{e^{2it(D-T_{u_0})}\}_{t\in\R}$ is the unitary group generated by the antiselfadjoint operator $2i(D-T_{u_0})$. Series expanding \eqref{eq: explicitFormula} in $z$, and using that $\lim_{r\uparrow 1}(\Pi u(t))(x)=F_{u(t)}(re^{ix})$, results in the following Fourier expansion of $\Pi u(t)$ 
\begin{equation}\label{eq: explicitFourierSeries}
    \big(\Pi u(t)\big)(x) = \sum_{k=0}^\infty e^{ikt}\big\langle\big(e^{2it(D-T_{u_0})}S^*\big)^{k}\Pi u_0,1\big\rangle e^{ikx} \quad \text{in }L^2_+(\T).
\end{equation}
The analogous explicit formula for the solution $u^{\varepsilon}$ of \eqref{eq: BOepsilon} can be obtained through a scaling argument: Note that $t\mapsto u^{\varepsilon}(t)$ solves \eqref{eq: BOepsilon} precisely when $t\mapsto  u^{\varepsilon}(t/\varepsilon)/\varepsilon$ solves \eqref{eq: BO}. That is, replacing $u_0$ by $u_0/\varepsilon$ in the series \eqref{eq: explicitFourierSeries} results in the Fourier expansion for $u^{\varepsilon}(t/\varepsilon)/\varepsilon$. And so, if we multiply each side by $\varepsilon$, make the change of variables $t\mapsto \varepsilon t$, and use that $T_{u_0/\varepsilon}=\varepsilon^{-1}T_{u_0}$, we similarly get
\begin{equation}\label{eq: explicitFormulaForBOEpsilon}
    \big(\Pi u^{\varepsilon}(t)\big)(x)=\sum_{k=0}^\infty e^{i\varepsilon kt}\big\langle\big(e^{2it(\varepsilon D-T_{u_0})}S^*\big)^{k}\Pi u_0,1\big\rangle e^{ikx}\quad \text{in }L^2_+(\T).
\end{equation}
From this formula, it is not difficult to show that  $(u^\varepsilon)_{\varepsilon>0}$ admits a weak-$L^2$ limit, which is what we do next. In what follows, $L^2_w(\T,\R)$ denotes $L^2(\T,\R)$ endowed with its weak topology, and we endow $C(\R, L^2_w(\T,\R))$ with the topology of compact convergence, that is, the topology generated by the seminorms
\begin{equation}\label{eq: theSeminormsOfCweakL2}
    |u|_{K,\varphi}\coloneqq \sup_{t\in K}|\langle u(t),\varphi\rangle|,
\end{equation}
where $K\subset \R$ is compact and $\varphi\in L^2(\T,\R)$. 
\begin{proposition}[Existence of a zero-dispersion limit]\label{prop: existenceOfZDLimit}
    Let $(u^\varepsilon)_{\varepsilon>0}\subset C(\R,L^2(\T,\R))$ be the solutions\footnote{In the sense of Molinet \cite{MolinetPeriodic}.} of \eqref{eq: BOepsilon} for some $u_0\in L^\infty(\T,\R)$. If $\varepsilon\to0$, then $u^\varepsilon$ converges in $C(\R, L^2_w(\T,\R))$, and the limiting function $ZD[u_0]$ is characterized by the Fourier expansion
    \begin{equation}\label{eq: explicitFourierExpansionOfZD}
       \Pi\big(ZD[u_0](t)\big)(x)= \sum_{k=0}^\infty\big\langle\big(e^{-2itT_{u_0}}S^*\big)^{k}\Pi u_0,1\big\rangle e^{ikx}\quad \text{in }L^2_+(\T),
    \end{equation}
    or, equivalently, by its $\mathcal{H}^2(\mathbb{D})$-representation $F_{ZD}(t,z)\coloneqq F_{ZD[u_0](t)}(z)$,
    \begin{equation}\label{eq: explicitFormulaForHardyRepresentationOfZeroDispersionWithoutCharacterization}
  F_{ZD}(t,z) = \big\langle\big(\mathrm{id}-ze^{-2itT_{u_0}}S^*\big)^{-1}\Pi u_0,1\big\rangle.
\end{equation}
\end{proposition}
\begin{proof}
   Since \eqref{eq: BO} preserves the $L^2$-norm, so does \eqref{eq: BOepsilon}. Thus $\|u^\varepsilon(t)\|_{L^2}=\|u_0\|_{L^2}$ for all $\varepsilon>0$ and $t\in \R$. Suppose for now that each Fourier coefficient of $u^\varepsilon(t)$ converges as $\varepsilon\to 0$ and that the convergence holds locally uniformly in time. It then follows that $\langle u^\varepsilon (t),\varphi\rangle$ converges whenever $\varphi$ is a trigonometric polynomial and, by density and the uniform bound on $\|u^\varepsilon(t)\|_{L^2}$, the same holds for any $\varphi\in L^2(\T)$. Thus, $u^\varepsilon(t)$ converges weakly in $L^2(\T)$ for each $t\in\R$ to a limit denoted $ZD[u_0](t)$. Moreover, by our assumption, this convergence holds locally uniformly in time in the sense that
   \begin{equation}\label{eq: convergenceInTheFunctionSpace}
           \lim_{\varepsilon\to 0}|u^\varepsilon - ZD[u_0]|_{K,\varphi}=0,
   \end{equation}
   for each compact $K\subset \R$ and $\varphi\in L^2(\T)$, where the seminorm is defined in \eqref{eq: theSeminormsOfCweakL2}. Note that \eqref{eq: convergenceInTheFunctionSpace} ensures that $t\mapsto ZD[u_0](t)$ is continuous in $L^2_w(\T)$ since the uniform limit of continuous functions is continuous.
   
The first part of the proposition thus follows if we can establish (uniformly in time) the $\varepsilon\to 0$ limit of the Fourier coefficients of $u^\varepsilon(t)$. Since each $u^\varepsilon(t)$ is real-valued, it suffices to do so for the Fourier coefficients corresponding to the nonnegative frequencies for which we can use the explicit formula \eqref{eq: explicitFormulaForBOEpsilon}.

For this purpose, consider the selfadjoint operators $\mathcal{L}^\varepsilon\coloneqq \varepsilon D - T_{u_0}$ on $L^2_+(\T)$ whose domain is $H^1_+(\T)\coloneqq H^1(\T,\C)\cap L^2_+(\T)$. We note the following two trivial facts: Firstly, 
    \begin{equation}\label{eq: strongConvergenceOnCore}
        \lim_{\varepsilon\to 0}\|(\mathcal{L}^\varepsilon +T_{u_0})f\|_{L^2_+(\T)}=0,\quad\text{for all }f\in H^1_+(\T).
    \end{equation}
Secondly, because $H^1_+(\T)\subset L^2_+(\T)$ is dense and $T_{u_0}$ is bounded on $L^2_+(\T)$, it follows that $H^1_+(\T)$ is a core for $-T_{u_0}$. Consequently,  $\mathcal{L}^\varepsilon\to -T_{u_0}$ in the strong resolvent sense \cite[Theorem VIII.25.(a)]{ReedSimon1} and so $g(\mathcal{L}^{\varepsilon})\to g(-T_{u_0})$ strongly for any $g\in C_{b}(\R,\C)$ \cite[Theorem VIII.20.(b)]{ReedSimon1}. In particular, $e^{2it\mathcal{L}^\varepsilon}\to e^{-2itT_{u_0}}$ strongly and, in fact, we claim that this convergence holds locally uniformly in time, that is
\begin{equation}\label{eq: uniformConvergenceInTime}
    \lim_{\varepsilon\to 0}\sup_{t\in K}\|(e^{2it\mathcal{L}^\varepsilon}- e^{-2itT_{u_0}})f\|_{L^2_+(\T)}=0,
\end{equation}
for every compact $K\subset \R$ and $f\in L^2_+(\T)$. By density, it suffices to prove \eqref{eq: uniformConvergenceInTime} when $f\in H_+^1(\T)$. For such $f$, we find that the $L^2_+(\T)$-valued functions $G_\varepsilon (t)\coloneqq e^{2it\mathcal{L}^\varepsilon}\!f$ are, for $\varepsilon\in[0,1]$, uniformly Lipschitz continuous because $\|G'_\varepsilon(t)\|_{L^2_+(\T)}= 2\|\mathcal{L}^\varepsilon f\|\leq 2\|Df\|_{L^2_+(\T)} + 2\|T_{u_0}f\|_{L^2_+(\T)}\eqqcolon 2L$.  Thus, for any sequence $(t_n,\varepsilon_n)\to (t,0)$ we obtain $\|G_{\varepsilon_n}(t_n)-G_{0}(t_n)\|_{L^2_+(\T)}\leq 2L|t_n-t| + \|G_{\varepsilon_n}(t)-G_{0}(t)\|_{L^2_+(\T)}\to 0$, which proves \eqref{eq: uniformConvergenceInTime}. This further implies that the following limits hold locally uniformly in time: 
\begin{equation*}%\label{eq: theLimitOfFourierCoefficients}
\lim_{\varepsilon\to 0}\Big(e^{2it\mathcal{L}^{\varepsilon}}S^*\Big)^k\Pi u_0=\Big(e^{-2itT_{u_0}}S^*\Big )^k\Pi u_0\quad \text{in }L^2_+(\T),\quad k\in\N_0.
\end{equation*}
From the discussion at the start of the proof, and the explicit formula \eqref{eq: explicitFormulaForBOEpsilon}, we conclude that $u^\varepsilon\to ZD[u_0]$ in $C(\R,L^2_w(\T))$, where $ZD[u_0]$ is characterized by \eqref{eq: explicitFourierExpansionOfZD}. Finally, expanding $(\mathrm{id}-ze^{-2itT_{u_0}}S^*)^{-1}$ as a Neumann series in $z$, one obtains \eqref{eq: explicitFourierExpansionOfZD} from \eqref{eq: explicitFormulaForHardyRepresentationOfZeroDispersionWithoutCharacterization} upon substituting $z\mapsto re^{ix}$ and letting $r\uparrow 1$. 
\end{proof}

\begin{remark}[On the restriction $u_0\in L^\infty(\T,\R)$]\label{rem: onTheLInfinityAssumption}
    Since both \eqref{eq: BO} and the mapping $u_0\mapsto ZD[u_0]$ are well-posed for $L^2$ data, one may naturally expect that Theorem \ref{thm: characterizingTheZeroDispersionLimit} extends to $u_0\in L^2(\T,\R)$. This could be true, but requires a different proof: If $u_0\in L^2(\T,\R)$ is unbounded, then the operator $T_{u_0}$ is also unbounded, as first shown by Toeplitz \cite{Toeplitz1,Toeplitz2}. This is unproblematic for the explicit formula since $T_{u_0}$ is still a well-defined symmetric operator on $H^1_+(\T)=\mathrm{dom}(D)$ which is ``infinitesimally small'' with respect to $D$, that is, for any $f\in H^1_+(\T)$ we have
    \begin{equation*}
        \|T_{u_0} f\|_{L^2_+(\T)}\leq  \|u_0f\|_{L^2(\T)}\leq \|u_0\|_{L^2(\T)}\|f\|_{L^\infty(\T)}\leq a\|Df\|_{L^2_+(\T)} + b\|f\|_{L^2_+(\T)}
    \end{equation*} 
    where $a>0$ can be made arbitrarily small provided $b>0$ (which is independent of $f$) is made sufficiently large.
 It follows by the Kato--Rellich theorem \cite[Theorem X.12]{ReedSimon2} that $D-T_{u_0}$ is selfadjoint on $H^1_+(\T)$. The same is true for $\varepsilon D-T_{u_0}$ for any $\varepsilon>0$, but the $\varepsilon\downarrow 0$-limit becomes ambiguous, because $T_{u_0}$ need not be essentially self-adjoint on $H^1_+(\T)$ \cite[Appendix A]{ToeplitzCounterExample}. In the previous proof, it is then no longer clear if the strong limit of $e^{2it\mathcal{L}^\varepsilon}$ exists as $\varepsilon\downarrow 0$ or, if so, what this limit is.

   On the other hand, we note that there exist unbounded $u_0\in L^2(\T,\R)$ such that $T_{u_0}$ is still  essentially selfadjoint on $H^1_+(\T)$. As pointed out in \cite{UnboundedToeplitz}, this holds whenever $u_0$ is locally semibounded, that is
    \begin{equation}\label{eq: thm1AppliesToSemiboundedData}
        \overline{u_0^{-1}((-\infty,-R])}\cap  \overline{u_0^{-1}([R,\infty))}=\emptyset,\quad \text{for sufficiently large }R>0.
    \end{equation}
For such $u_0$, Theorem \ref{thm: characterizingTheZeroDispersionLimit} applies --- the proof being the same as for bounded $u_0$.
\end{remark}

\begin{remark}[Letting the initial data depend on $\varepsilon$]\label{rem: initialDataCanDependOnEpsilon}
    In the previous proof, the steps would be exactly the same if the initial data was $\varepsilon$-dependent, i.e.~ if the family of solutions $(u^\varepsilon)_{\varepsilon>0}$ had initial data $(u_0^{\varepsilon})_{\varepsilon>0}\subset L^2(\T,\R)$ such that the following $L^2$-limit held, $\lim_{\varepsilon\to 0}u_0^\varepsilon=u_0\in L^\infty(\T,\R)$. Then $\mathcal{L}^\varepsilon= \varepsilon D - T_{u_0^\varepsilon}$ would still converge strongly to $-T_{u_0}$ on $H^1_+(\T)$. Proposition \ref{prop: existenceOfZDLimit} thus holds true for $\varepsilon$-dependent initial data, and so does Theorem \ref{thm: characterizingTheZeroDispersionLimit}.
\end{remark}

\subsection{The generalization of Miller--Xu's alternating sum formula}
The characterization \eqref{eq: alternatingSumFormula} of $ZD[u_0]$ by Miller and Xu, requires that $u_{B}$ is a finite set for a.e.~$(t,x)\in \R\times \T$. This is true for sufficiently smooth $u_0$ (finite variation will do), but in general false for $u_0\in L^\infty(\T,\R)$. We here introduce an appropriate generalization of their formula.

Consider for now $u_{0}\in C^1(\T,\R)$, and let $u_B$ denote the corresponding classical solution of Burgers' equation \eqref{eq: burgersEquation} which exists for sufficiently small times $|t|\ll 1$. In the spirit of the kinetic formulation for hyperbolic conservation laws \cite{Kinetic} we introduce the $\chi$-function
\begin{equation*}
        \chi(u,y)\coloneqq  \mathbbm{1}_{u>y>0} -\mathbbm{1}_{u<y<0},\qquad u,y\in\R,
\end{equation*}
and write for brevity $\chi_B\coloneq \chi\circ u_B=\chi(u_B(\cdot,\cdot),\cdot)$. Intuitively, $\chi_B(t,\cdot,\cdot)$ represents the $(x,y)$-region enclosed by the graph $\{(x,y)\colon y=u_B(t,x)\}$ and the line $y=0$. Inside this region, $\chi_B(t,x,y)=\pm 1$ depending on $y\gtrless 0$, while outside it $\chi_B(t,x,y)=0$; see Figure \ref{fig: multivaluedBurgers}. The transformation $u_B\mapsto \chi_B$ is invertible, because $u_B(t,x)=\int_{\R}\chi(u_B(t,x),y)\, dy$, and has the benefit that \eqref{eq: burgersEquation} reduces to a linear transport equation (interpreted in the weak sense)
\begin{equation*}%\label{eq: theExtensionOfChi}
    \partial_t \chi_B + 2y\partial_x\chi_B = 0,\quad \implies \quad \chi_B(t,x,y)=\chi_B(0,x-2ty,y)=\chi(u_0(x-2ty),y),
\end{equation*}
where the argument $x-2ty$ is understood modulo $2\pi$. That is, for as long as $u_B$ is a classical solution of \eqref{eq: burgersEquation}, it is explicitly given by $u_B(t,x)=\int_{\R}\chi(u_0(x-2ty),y)dy$. When $u_B$ becomes multivalued, this last identity generalizes to \begin{equation}\label{eq: theGeneralizedMillerXuFormula}
   \sum_{n=0}^{N}(-1)^nu^n_B(t,x)= \int_{\R}\chi(u_0(x-2ty),y)\, dy ,
\end{equation}
with $u_B(t,x)=\{u_B^0(t,x)<u_B^1(t,x)<\dots<u_B^N(t,x)\}$ and where $(t,x)\in \R\times \T$ must avoid the null set of caustics arising from the characteristics $y=u_0(x-2ty)$. A proof of this identity is given in the appendix and it is illustrated in Figure \ref{fig: multivaluedBurgers}. 
\begin{figure}[h]
  \centering
  \includegraphics[width=\linewidth, trim=-1cm 0cm -1cm 0cm, clip]{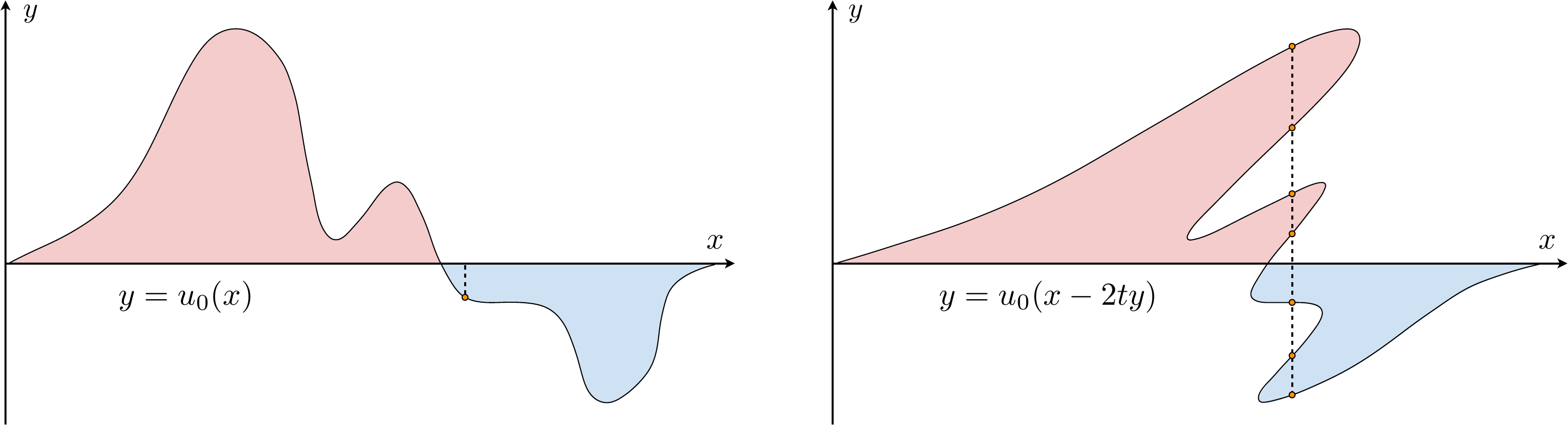}
  \caption{The graph (black curve) of the multivalued solution $u_B=u_0(x-2t u_B)$ of Burgers' equation, shown at $t=0$ and at a later time $t>0$. The red region corresponds to $\chi(u_0(x-2ty),y)=1$, while the blue region corresponds to $\chi(u_0(x-2ty),y)=-1$. The dashed vertical line intersects the graph once initially and seven times at the later time; the alternating sum of these seven $y$-values coincides with the integral of $\chi(u_0(x-2ty),y)$ along the dashed line.}
  \label{fig: multivaluedBurgers}
\end{figure}
The benefit of \eqref{eq: theGeneralizedMillerXuFormula} is that the right-hand side extends to all data $u_0\in L^\infty(\T,\R)$. Naturally, we expect it to coincide with the zero-dispersion limit from Proposition \ref{prop: existenceOfZDLimit}, that is, we aim to prove that
\begin{equation*}%\label{eq: theCharacterizationOfZDInCleanForm}
    ZD[u_0](t,x) =  \int_{\R}\chi(u_0(x-2ty),y)\, dy.
\end{equation*}
Since we know the Fourier series of $ZD[u_0]$, we compute this for the right-hand side too:
\begin{proposition}[Generalized alternating sum formula]\label{prop: generalizedASFormula}
 For $u_0\in L^\infty(\T,\R)$ and $t\in \R$, consider the function $AS[u_0](t)\in L^\infty(\T,\R)$ defined by
 \begin{equation*}
     AS[u_0](t,x)\coloneqq \int_{\R}\chi(u_0(x-2ty),y)\, dy,\qquad \chi(u,y)\coloneqq \mathbbm{1}_{u>y>0} -\mathbbm{1}_{u<y<0},\quad u,y\in\R,
 \end{equation*}
where the argument $x-2ty$ is understood modulo $2\pi$. Then $\Pi\big(AS[u_0](t)\big)$ admits the Fourier expansion 
 \begin{equation}\label{eq: theFourierSeriesOfAS}
    \Pi\big(AS[u_0](t)\big)(x)= \sum_{k=0}^\infty \bigg\langle\frac{1-e^{-2itku_0}}{2itk},q^k\bigg\rangle e^{ikx} \quad \text{in }L^2_+(\T),
 \end{equation}
 where $q(x)= e^{ix}$, and where the coefficients must be interpreted in limit sense when $tk=0$. Equivalently, $AS[u_0](t)$ can be represented in $\mathcal{H}^2(\mathbb{D})$ by $F_{AS}(t,z)\coloneqq F_{AS[u_0](t)}(z)$ which reads
 \begin{equation}\label{eq: theHardyElementOfAS}
         F_{AS}(t,z)= \frac{1}{2\pi}\int_{\T}u_0(x)\, dx + \frac{1}{4\pi it}\int_{\T}\mathrm{Log}\bigg(\frac{1-ze^{-i(x+2tu_0(x))}}{1-ze^{-ix}}\bigg)\, dx,
 \end{equation}
 where $\Log$ denotes the principal-branch logarithm.
\end{proposition}
\begin{remark}
    The $\Log$-term is well defined because $w\in \mathbb{D}\implies\mathrm{Re}[1-w]>0$, which means that
    \begin{equation*}
        \frac{1-w_1}{1-w_2}\in \C\setminus(-\infty,0],\quad \forall w_1,w_2\in \mathbb{D}.
    \end{equation*}
We also point out that the analogous expression for \eqref{eq: theHardyElementOfAS} on the real line \cite{ZeroDispersionLimitGerard,ZeroDispersionLimitXiChen} is
\begin{equation}\label{eq: XiChensASFormula}
    (t,\tilde z)\mapsto \frac{1}{4\pi it}\int_{\R}\Log\bigg(1+\frac{2t u_0(x)}{x-\tilde z}\bigg)\, dx,
\end{equation}
with domain $(t,\tilde{z})\in\R\times \C_+$ and $\C_+\coloneqq \{w\in \C\colon \mathrm{Im}[w]>0\}$.
It is indeed possible to obtain \eqref{eq: XiChensASFormula} as a long-period-limit of \eqref{eq: theHardyElementOfAS} (as one might expect), but we will not do so here.
\end{remark}
\begin{proof}[Proof of Proposition \ref{prop: generalizedASFormula}]
For $k\geq 0$, we compute 
\begin{align*}
    \mathcal{F}_x\Bigg[\int_{\R}\chi(u_0(x-2ty),y)\, dy\Bigg](k)
    =&\, \frac{1}{2\pi}\int_{\R}\int_{\T}\chi(u_0(x-2ty),y)e^{-ikx}\, dx\,dy\\ =&\, \frac{1}{2\pi}\int_{\R}\int_{\T}\chi(u_0(x),y)e^{-ik(x+2ty)}\, dx\, dy\\
    =&\,  \frac{1}{2\pi}\int_{\T}\bigg(\int_0^{u_0(x)}e^{-2itky}\, dy \bigg)e^{-ikx}\, dx\\
    =&\, \frac{1}{2\pi}\int_{\T}\bigg(\frac{1-e^{-2itku_0(x)}}{2itk}\bigg)e^{-ikx}\, dx
     =  \bigg\langle\frac{1-e^{-2itku_0}}{2itk},q^k\bigg\rangle.
\end{align*}
Thus, \eqref{eq: theFourierSeriesOfAS} is clear. And since  
\begin{equation*}
    \mathrm{Log}\bigg(\frac{1-z\overline{q}e^{-2itu_0}}{1-z\overline{q}}\bigg) = \mathrm{Log}\Big(\frac{1}{1-z\overline{q}}\Big) - \mathrm{Log}\Big(\frac{1}{1-z\overline{q}e^{-2itu_0}}\Big) = \sum_{k=1}^\infty \frac{(z\overline{q})^{k}}{k}\big(1-e^{-2itku_0}\big),
\end{equation*}
we see that \eqref{eq: theHardyElementOfAS} leads to \eqref{eq: theFourierSeriesOfAS} when we set $z\mapsto re^{ix}$ and let $r\uparrow 1$.
\end{proof}
\section{Proving the main results}\label{sec: provingTheMainResults}
This section is split into three subsections, each devoted to proving Theorem \ref{thm: characterizingTheZeroDispersionLimit}, Lemma \ref{lem: theRaney-typeLemma}, and Theorem \ref{thm: propertiesOfZD}, respectively.
\subsection{Characterizing the zero-dispersion limit}
In this section we prove Theorem \ref{thm: characterizingTheZeroDispersionLimit}. In light of Propositions \ref{prop: existenceOfZDLimit} and \ref{prop: generalizedASFormula}, it remains only to show that $ZD[u_0]=AS[u_0]$ (equivalently $F_{ZD}=F_{AS}$). We do so by comparing the respective Fourier coefficients in the series \eqref{eq: explicitFourierExpansionOfZD} and \eqref{eq: theFourierSeriesOfAS}, which is done by expanding these coefficients in $t$. The latter expansions will only be valid for $u_0\in H^1(\T,\R)$, but Theorem \ref{thm: characterizingTheZeroDispersionLimit} can afterwards be extended to $u_0\in L^\infty(\T,\R)$ by a density argument. 
\begin{proposition}[Series expanding the Fourier coefficients]\label{prop: seriesExpansions}
    Let $k\in \N$, $t\in\R$, and $u_0\in H^1(\T,\R)$. Then we have the following series expansions written in notation introduced below or in Section \ref{sec: preliminaries}:
    \begin{align}
    \label{eq: theHardExpansion}
\Big\langle\big(e^{-2it T_{u_0}}S^*\big)^k \Pi u_0,1\Big\rangle =&\,\sum_{d=0}^\infty\Bigg(\sum_{\substack{m\in  \mathcal{M}_{d,k}\\ n\in \mathcal{N}_{d,k}}}\frac{S(m,n)}{n!}\hat{u}_0(m)\Bigg)(-2it)^{d},\\
    \label{eq: theEasyExpansion}
\Big\langle\frac{1-e^{-2itku_0}}{2itk},q^k\Big\rangle=&\, \sum_{d=0}^\infty\Bigg(\frac{k^d}{(d+1)!}\sum_{m\in \mathcal{M}_{d,k}}\hat{u}_0(m) \Bigg)(-2it)^{d}.
    \end{align}
The two sets $\mathcal{M}_{d,k}$ and $\mathcal{N}_{d,k}$ are defined
\begin{equation*}%\label{eq: definitionOfTheSets}
        \mathcal{M}_{d,k}\coloneqq  \Big\{m\in \Z^{d+1}\colon {\textstyle\sum_{j=1}^{d+1}m_j=k}\Big\},\qquad \mathcal{N}_{d,k}\coloneqq  \Big\{n\in \N_0^k\colon {\textstyle \sum_{\ell=1}^{k}n_\ell=d}\Big\},
\end{equation*}
and we have employed the multi-index notation 
\begin{equation*}
    \hat{u}_0(m)\coloneqq\hat{u}_0(m_1)\cdots \hat{u}_0(m_{d+1})\qquad n!\coloneqq n_1!\cdots n_{k}!.
\end{equation*}
Moreover, $S(m,n)$ is the following composition of shifts applied to unity 
\begin{equation}\label{eq: theCompositionOfShiftsWrittenEasier}
    S(m,n)= \Bigg[\Bigg(\prod_{j=1}^{N_1}S_{m_j}\Bigg) S^*\Bigg]\Bigg[\Bigg(\prod_{j=N_1+1}^{N_2}S_{m_j}\Bigg) S^*\Bigg]\cdots\Bigg[\Bigg(\prod_{j=N_{k-1}+1}^{N_k}S_{m_j}\Bigg) S^*\Bigg]S_{m_{d+1}}(1),
\end{equation}
where $N_{\ell}\coloneqq n_1+\dots+n_{\ell}$ and $N_{0}\coloneqq 0$, and where $S_y\coloneqq T_{q^y}$ with $q(x)= e^{ix}$. The indexed products in \eqref{eq: theCompositionOfShiftsWrittenEasier} are ordered, $\prod_{j=a}^b c_j\coloneqq c_a\cdots c_b$, with the convention $\prod_{j=a}^bc_j\coloneqq \id$ if $a>b$.
\end{proposition}
\begin{remark}[Interpreting the weight $S(m,n)$]\label{rem: interpretingTheWeight}
From \eqref{eq: theCompositionOfShiftsWrittenEasier} it is not immediately evident that $S(m,n)\in\{0,1\}$, and so we explain this in detail. Each of the $k$ left-shifts that appear in \eqref{eq: theCompositionOfShiftsWrittenEasier} can be written $S^*=S_{-1}$, and so we see that $S(m,n)$ is a large composition of shifts $S_y$ for various values of $y\in\Z$. Moreover, the sum of all these $y$-values is necessarily $m_1+\dots +m_{d+1} - k =0$. By \eqref{eq: explicitCompositionOfToeplitzShifts} from Lemma \ref{lem: compositonsOfToeplitzShifts} it follows that $S(m,n)\in\{0,1\}$. However, determining from $(m,n)$ whether the condition in \eqref{eq: explicitCompositionOfToeplitzShifts} holds---and hence which value $S(m,n)$ takes---is somewhat involved; see Figure~\ref{fig: diagramsExample} for an illustration.
\end{remark}
\begin{proof}[Proof of Proposition \ref{prop: seriesExpansions}]
Replacing $\Pi u_0$ by $T_{u_0}(1)$ on the left-hand side in \eqref{eq: theHardExpansion}, we have the following expansion of the resulting operator  
    \begin{align}\label{eq: simpleExpansionsTowardsHard}
        \Big(e^{-2it T_{u_0}}S^*\Big)^kT_{u_0}= \sum_{n\in \N_0^k} \frac{(-2it)^{d}}{n!}(T_{u_0})^{n_1}S^*\cdots (T_{u_0})^{n_k}S^* T_{u_0},
    \end{align}
where $d\coloneqq n_1+\dots +n_k$ and $n!=n_1!\cdots n_k!$. Since the Fourier series of $u_0$ is absolutely convergent, powers of $T_{u_0}$ are weighted, and absolutely convergent, sums of shift-compositions; for example
\begin{align*}
      (T_{u_0})^{n_1} = \sum_{m\in \Z^{n_1}}\hat{u}_0(m) S_{m_1}\cdots S_{m_{n_1}},
\end{align*}
where, as before, we use the short-hand notation $\hat{u}_0(m)=\hat{u}_0(m_1)\cdots\hat{u}_0(m_{n_1})$.
By similarly expanding $(T_{u_0})^{n_2},\dots,(T_{u_0})^{n_k}$ and $T_{u_0}$ in \eqref{eq: simpleExpansionsTowardsHard}, and then concatenating the $m$-vectors, we obtain the identity
\begin{equation*}
     \Big(e^{-2it T_{u_0}}S^*\Big)^kT_{u_0}(1)= \sum_{n\in \N_0^k} \sum_{m\in \Z^{d+1}}\frac{S(m,n)}{n!}\hat{u}_0(m)(-2it)^{d},
\end{equation*}
where $S(m,n)$ is defined as in \eqref{eq: theCompositionOfShiftsWrittenEasier} (without any restriction on the sum of the entries in $m$). Thus
\begin{align}\label{eq: closeToTheEndOfTheHardExpansion}
   \Big\langle\big(e^{-2it T_{u_0}}S^*\big)^kT_{u_0}(1),1\Big\rangle =\sum_{n\in \N_0^k}\sum_{m\in \Z^{d+1}} \frac{\langle S(m,n),1\rangle}{n!}\hat{u}_0(m)(-2it)^{d}.
\end{align}
Reasoning similarly as in Remark \ref{rem: interpretingTheWeight}, one sees that $S(m,n)\in\{0,q^s\}$ where $s\coloneqq m_1+\dots +m_{d+1} -k$.
Thus, $\langle S(m,n),1\rangle=0$ whenever $s\neq 0$ and so one may in \eqref{eq: closeToTheEndOfTheHardExpansion} restrict $m\in \mathcal{M}_{d,k}$ (that is, when $s=0$). Moreover, since $S(m,n)\in\{0,1\}$ for such $m$, we can replace $\langle S(m,n),1\rangle$ by $S(m,n)$. With these modifications, one obtains \eqref{eq: theHardExpansion} from \eqref{eq: closeToTheEndOfTheHardExpansion}. 
We skip the proof of \eqref{eq: theEasyExpansion} as it is similar to, and easier than, that of \eqref{eq: theHardExpansion}. 
\end{proof}

The only obstacle in comparing the two series \eqref{eq: theEasyExpansion} and \eqref{eq: theHardExpansion} is the cumbersome weight $S(m,n)$. One could hope that simply summing over $n\in \mathcal{N}_{d,k}$ in \eqref{eq: theHardExpansion} resulted in \eqref{eq: theEasyExpansion}. However, it is typically \textit{not} the case that 
\begin{equation*}
    \sum_{n\in \mathcal{N}_{d,k}}\frac{S(m,n)}{n!}=\frac{k^d}{(d+1)!},
\end{equation*}
which can, for example, be seen when $m_{d+1}\leq 0$ in which case $S(m,n)=0$ for any $n\in \N_0^k$. Thus, a different argument is needed.

It turns out that summing $S(m,n)$ over the various cyclic permutations of $m$ and $n$ yields $k$, as demonstrated by the following Lemma \ref{lem: theRaney-typeLemma}. Conveniently, the remaining $(m,n)$-dependent terms in \eqref{eq: theHardExpansion} are invariant under permutations of $m$ and $n$. 

In what follows, $m\in \Z^{d+1}$ and $n\in \N_0^k$ will be endowed with cyclic indexing so that
\begin{equation*}
    m_{r(d+1)+j}= m_j,\qquad
     n_{r k+\ell}= n_\ell,
\end{equation*}
for $r\in \Z$, $j\in \{1,\dots,d+1\}$, and $\ell\in \{1,\dots,k\}$, and $T$ will denote the cyclic shift \begin{equation*}
    (Tm)_j=m_{j+1},\qquad (Tn)_\ell = n_{\ell +1}.
\end{equation*}
\begin{lemma}[A Raney–type lemma]\label{lem: theRaney-typeLemma}
    Let $m\in \Z^{d+1}$ and $n\in \N_0^k$ be such that $m_1+\dots +m_{d+1}=k\geq 1$ and $n_1+\dots +n_{k}=d\geq 0$. Then
    \begin{equation*}
    \sum_{j=1}^{d+1}\sum_{\ell =1}^{k} S(T^jm,T^\ell n)=k,
\end{equation*}
where $S(\cdot,\cdot)$ is as defined in \eqref{eq: theCompositionOfShiftsWrittenEasier}. That is, there are precisely $k$ pairs $(j,\ell)\in\{1,\dots,d+1\}\times \{1,\dots,k\}$ such that $ S(T^jm,T^\ell n)=1$. For the remaining pairs $(j,\ell)$, $S(T^jm,T^\ell n)$ is zero.
\end{lemma}
We prove this lemma in Section \ref{sec: combinatorics} by reformulating it into Lemma \ref{lem: RaneysLemma} known as Raney's lemma\footnote{It is named so in \cite{ConcreteMathematics} and appeared first in \cite{RaneysPaper}.}. To illustrate Lemma \ref{lem: theRaney-typeLemma}, we provide the following example which is related to Figure \ref{fig: diagramsExample}.
\begin{figure}[h]
\centering
    \includegraphics[width=0.9\linewidth, trim=0cm 0cm 0cm 0cm, clip]{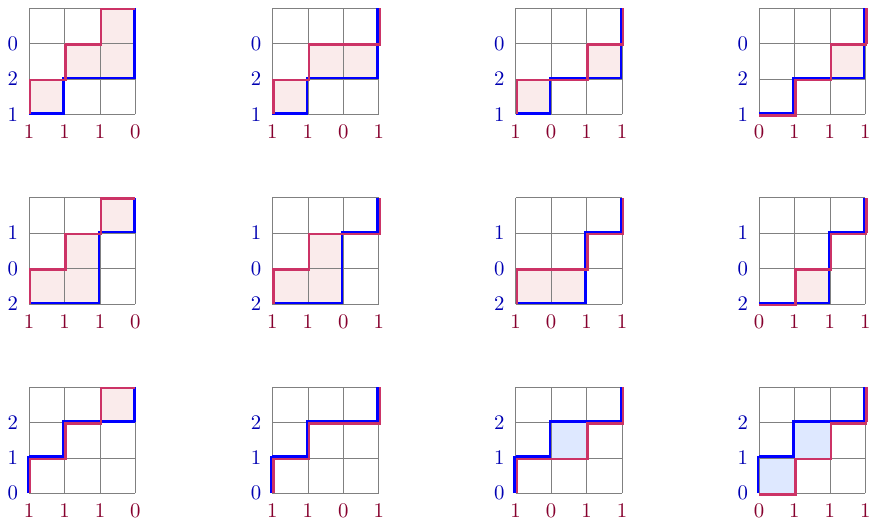}
     \caption{A visual representation of the twelve weights $S(T^j m,T^\ell n)$ resulting from $m=(0,1,1,1)$ and $n=(0,1,2)$. The diagram at column $j$ and row $\ell$ corresponds to $S(T^j m,T^\ell n)$. The red and blue line represents the partial sums of $T^j m$ and $T^\ell n$ respectively; the former grows vertically and the latter horizontally. Moreover, $S(T^j m,T^\ell n)=1$ if the red line is below the blue, and $S(T^j m,T^\ell n)=0$ otherwise, as justified by Lemma \ref{lem: aSimplerCharacterizationOfWhenTheCompositionOfShiftsIsOne}. Here, we see that $S(T^j m,T^\ell n)=1$ when $(j,\ell)\in\{2,3,4\}\times\{3\}$.}
    \label{fig: diagramsExample}
\end{figure}
\begin{remark}[An example of the Raney-type lemma]
    Set $d=k=3$, $m=(0,1,1,1)$ and $n=(0,1,2)$. By Lemma \ref{lem: theRaney-typeLemma}, there are exactly three pairs $(j,\ell)\in\{1,2,3,4\}\times\{1,2,3\}$ such that $S(T^jm,T^\ell n)=1$. From Figure \ref{fig: diagramsExample} we see that these are $(2,3)$, $(3,3)$, and $(4,3)$, which we identify, for simplicity, with $(2,0)$, $(3,0)$, and $(0,0)$, since $T^4m=m$ and $T^3n=n$. Writing the corresponding weights out using \eqref{eq: theCompositionOfShiftsWrittenEasier}, we indeed find
    \begin{align*}
 S(m,n)=S^*S_{m_1}S^*S_{m_2}S_{m_3}S^*S_{m_4}(1)=S_{-1}S_0S_{-1}S_1S_1S_{-1}S_1(1)=1,\\
    S(T^2m,n)=S^*S_{m_3}S^*S_{m_4}S_{m_1}S^*S_{m_2}(1)=S_{-1}S_1S_{-1}S_1S_0S_{-1}S_1(1)=1,\\
    S(T^3m,n)=S^*S_{m_4}S^*S_{m_1}S_{m_2}S^*S_{m_3}(1)=S_{-1}S_1S_{-1}S_0S_1S_{-1}S_1(1)=1.
\end{align*}
where we use $S^*=S_{-1}$ and Lemma \ref{lem: compositonsOfToeplitzShifts}. For any other pair $(j,\ell)\in\{1,2,3,4\}\times\{1,2,3\}$, one would find that $S(T^jm,T^\ell n)=0$. 
\end{remark}
We can now prove our main theorem:
\begin{proof}[Proof of Theorem \ref{thm: characterizingTheZeroDispersionLimit}]
    As explained at the start of the section, we need only prove that $ZD[u_0]$ from Proposition \ref{prop: existenceOfZDLimit} coincides with $AS[u_0]$ from Proposition \ref{prop: generalizedASFormula}. We first do this for $u_0\in H^1(\T,\R)$ by turning the right-hand side of \eqref{eq: theHardExpansion} into that of \eqref{eq: theEasyExpansion} (the $k=0$ case is clear; both the left-hand sides of \eqref{eq: theHardExpansion} and \eqref{eq: theEasyExpansion} become $\langle u_0,1\rangle=\langle \Pi u_0,1\rangle$).
    
   Note that the cyclic shift $T$ acts as a permutation on the sets $\mathcal{M}_{d,k}$ and $\mathcal{N}_{d,k}$. Using this, Lemma \ref{lem: theRaney-typeLemma}, and the multinomial theorem, we can manipulate the bracket on the right-hand side of \eqref{eq: theHardExpansion} as follows
   \begin{align*}
       \sum_{\substack{m\in  \mathcal{M}_{d,k}\\ n\in \mathcal{N}_{d,k}}}S(m,n)\frac{\hat{u}_0(m)}{n!} =&\, \frac{1}{(d+1)k}\sum_{\substack{m\in  \mathcal{M}_{d,k}\\ n\in \mathcal{N}_{d,k}}}\sum_{j=1}^{d+1}\sum_{\ell =1}^{k}S(T^jm,T^\ell n)\frac{\hat{u}_0(T^j m)}{(T^{\ell}n)!}\\
       =&\, \frac{1}{(d+1)k}\sum_{\substack{m\in  \mathcal{M}_{d,k}\\ n\in \mathcal{N}_{d,k}}}\sum_{j=1}^{d+1}\sum_{\ell =1}^{k}S(T^jm,T^\ell n)\frac{\hat{u}_0(m)}{n!}\\
            =&\, \frac{1}{(d+1)}\sum_{n\in \mathcal{N}_{d,k}}\frac{1}{n!}\sum_{m\in  \mathcal{M}_{d,k}}\hat{u}_0(m) = \frac{k^d}{(d+1)!}\sum_{m\in  \mathcal{M}_{d,k}}\hat{u}_0(m).
   \end{align*}
Thus, the two series \eqref{eq: theHardExpansion} and \eqref{eq: theEasyExpansion} agree and Theorem \ref{thm: characterizingTheZeroDispersionLimit} has been proved for $u_0\in H^1(\T,\R)$.

We extend the theorem to $u_0\in L^\infty(\T,\R)$ by a density argument: Pick elements $v_n\in H^1(\T,\R)$ such that $\lim_{n\to\infty} v_n= u_0$ in the $L^2(\T,\R)$ sense. It follows that $T_{v_n}\to T_{u_0}$ strongly on $H^1_+(\T)\subset L^2_+(\T)$ and, arguing as in the proof of Proposition \ref{prop: existenceOfZDLimit}, we get that $ZD[v_n](t)\to ZD[u_0](t)$ weakly in $L^2(\T,\R)$ for all $t\in \R$. For $\varphi\in L^2(\T,\R)$, we then infer
\begin{equation*}
    \langle ZD[u_0](t),\varphi \rangle = \lim_{n\to\infty}\langle ZD[v_n](t),\varphi \rangle=\lim_{n\to\infty}\langle AS[v_n](t),\varphi \rangle=\langle AS[u_0](t),\varphi \rangle,
\end{equation*}
where we used the $L^2$-continuity of $AS[\cdot](t)$ that is proved in Section \ref{sec: propertiesOfTheZeroDispersionLimit}. As $\varphi\in L^2(\T,\R)$ was arbitrary, we conclude that also $ZD[u_0](t)=AS[u_0](t)$.
\end{proof}

\subsection{Proving the Raney-type lemma}\label{sec: combinatorics}
Throughout the subsection, all finite sequences will be endowed with cyclic indexing. For a $\Z$-labeled sequence of real numbers $y=(\dots,y_{-1},y_0,y_1,\dots)$ and $j\in \Z$ we introduce the signed sum
\begin{equation*}
    \Sigma(y,j)\coloneqq \begin{cases}
        y_1+\dots+y_j, & j>0,\\
        -(y_{j+1}+\dots + y_0), & j<0,\\
         0, & j=0,
    \end{cases}
\end{equation*}
which is easily seen to satisfy
\begin{equation}\label{eq: theAlgebraPropertyOfSum}
     \Sigma(y,j_1+j_2)= \Sigma(y,j_1) + \Sigma(T^{j_1}y,j_2),\quad\forall j_1,j_2\in\Z. 
\end{equation}
We will prove Lemma \ref{lem: theRaney-typeLemma} by reducing it to the following Raney’s lemma. 
\begin{lemma}[Raney's lemma, \cite{ConcreteMathematics}]\label{lem: RaneysLemma}
  If $y\in \Z^{J}$ sums to one, that is, $\Sigma(y,J)=1$, then there is a unique cyclic shift $T^r\in\{T,\dots, T^{J}\}$ such that the partial sums of $T^r y$ are all positive: 
  \begin{equation}\label{eq: theRaneyCriterion}
      \Sigma(T^ry,j)>  0,\qquad j\in \{1,\dots,J\}.
  \end{equation}
\end{lemma}
\begin{proof}
     Define the $\R$-valued sequence $x_j\coloneqq y_j- \frac{1}{J}$. Since $y$ is integer valued, we infer two things: First, we have the equivalence $\min_{1\leq j\leq J}\Sigma(T^ry,j)>0$ if, and only if, $\min_{1\leq j\leq J}\Sigma(T^rx,j)\geq 0$. Second, the numbers
$\Sigma(x,j)=\Sigma(y,j)-\frac{j}{J}$
     are all different for $j\in\{1,\dots,J\}$, and so there is a unique $r\in\{1,\dots, J\}$ that minimizes $j\mapsto \Sigma(x,j)$; we claim that $T^r$ is the unique shift that satisfies \eqref{eq: theRaneyCriterion}. To see this, we first infer from \eqref{eq: theAlgebraPropertyOfSum} that
        \begin{equation}\label{eq: partialSumForAShift}
        \Sigma(T^sx,j) =  \Sigma(x,s+j) - \Sigma(x,s),
    \end{equation}
for any $s,j\in\{1,\dots, J\}$. One should note $j\mapsto \Sigma(x,j)$ is $J$-periodic since $x$ sums to zero. Now, if $s=r$ then the right-hand side of \eqref{eq: partialSumForAShift} is nonnegative for all $j\in \{1,\dots,J\}$, and if $s\neq r$ then it is negative for the $j\in \{1,\dots,J\}$ satisfying $s+j=r$ mod$\,J$. By the equivalence established at the start of the proof, we are done.
\end{proof}
Actually, one may view this lemma as a special case of Lemma \ref{lem: theRaney-typeLemma}. Indeed, if $y\in\Z^J$ is as in Lemma \ref{lem: RaneysLemma}, then $m\in \Z^{J}$ and $n\in\N_{0}^{J-1}$, defined $m_j\coloneqq 1-y_j$ and $n\coloneqq (1,1,\dots,1)$, satisfies the assumptions of Lemma \ref{lem: theRaney-typeLemma}. Moreover, it is not hard to verify that $S(m,n)=1$ if, and only if, $y$ satisfies \eqref{eq: theRaneyCriterion}. Then, since $T^\ell n=n$ for all $\ell\in\Z$, one sees that the Lemmas \ref{lem: theRaney-typeLemma} and \ref{lem: RaneysLemma} coincide in this case. This correspondence will be established in greater generality in Lemma \ref{lem: theGoodPropertiesOfTheRaneySequence}.
A first step in this direction is to obtain a more concrete characterization of when $S(m,n)$ is one/zero:
  \begin{lemma}\label{lem: aSimplerCharacterizationOfWhenTheCompositionOfShiftsIsOne}
    Let $d\geq 0$, $k\geq 1$, and $(m,n)\in \mathcal{M}_{d,k}\times \mathcal{N}_{d,k}$, where the latter sets are as in Proposition \ref{prop: seriesExpansions}. Then, we have the equivalence
    \begin{equation}\label{eq: theCharacterizationOfWeightIsOne}
   S(m,n) =1,\qquad\quad \Longleftrightarrow \qquad\quad \Sigma(n,\Sigma (m,j))<j,\,\,\,\forall j\in\{1,\dots,d+1\},
\end{equation}
and otherwise $S(m,n)=0$.
\end{lemma}
\begin{proof}
 By the first part of Lemma \ref{lem: compositonsOfToeplitzShifts}, we infer that the $k$ left shifts $S^*=S_{-1}$ that appear in \eqref{eq: theCompositionOfShiftsWrittenEasier} can be merged with the shifts that follow them; for example, the rightmost left-shift can be merged with $S_{m_{d+1}}$, yielding $S_{m_{d+1}-1}$. In doing so, we get
\begin{equation*}
    S(m,n)=S_{y_1}S_{y_2}\cdots S_{y_{d+1}}(1),\qquad y_j\coloneqq m_j - |\{\ell\in\{1,\dots,k\}\colon \Sigma(n,\ell) = j-1\}|,
\end{equation*}
where $|\cdot|$ denotes cardinality. Now, because $y_1+\dots+y_{d+1}=m_1+\dots+m_{d+1}-k=0$, we get the equivalence $y_1+\dots+y_j\leq 0 \Longleftrightarrow y_{j+1}+\dots+y_{d+1}\geq 0$ for every $j\in\{1,\dots,d+1\}$. With this in mind, we use Lemma \ref{lem: compositonsOfToeplitzShifts} to conclude that $ S(m,n)=1$ if
\begin{equation}\label{eq: theFirstTrivialCondition}
y_1+\dots+y_j\leq 0,\qquad\forall j\in\{1,\dots,d+1\},
\end{equation}
and $S(m,n)=0$ else. Inserting for the $y$-values in \eqref{eq: theFirstTrivialCondition} we find that the condition is equivalent to
\begin{equation}\label{eq: theSecondTrivialConditon}
     \Sigma(m,j)\leq \max\{\ell\in\{0,\dots,k\}\colon \Sigma(n,\ell) < j\},\qquad\forall j\in\{1,\dots,d+1\},
\end{equation}
where the use of $\max\{\cdot\}$ is appropriate since $\ell\mapsto \Sigma(n,\ell)$ is nondecreasing. The restriction $\ell\in\{0,\dots,k\}$ can be relaxed to $\ell\in\Z$; that is, we claim that \eqref{eq: theSecondTrivialConditon} is equivalent to
\begin{equation}\label{eq: theThirdTrivialConditon}
     \Sigma(m,j)\leq \max\{\ell\in\Z\colon \Sigma(n,\ell) < j\},\qquad \forall j\in\{1,\dots,d+1\}.
\end{equation}
Indeed, if $j=d+1$ then both conditions are trivially satisfied since $\Sigma(m,d+1)=k$ and $\Sigma(n,k)=d$, and if $j\in\{1,\dots,d\}$ then the right-hand sides of \eqref{eq: theSecondTrivialConditon} and \eqref{eq: theThirdTrivialConditon} coincide. Finally, observe that 
\begin{equation}\label{eq: theObviousEquivalence}
\Sigma(n,p)<j \quad \Longleftrightarrow\quad p\leq \max\{\ell\in\Z\colon \Sigma(n,\ell) < j\},
\end{equation}
for all $p,j\in \Z$, since $\ell\mapsto \Sigma(n,\ell)$ is nondecreasing. Setting $p=\Sigma(m,j)$, \eqref{eq: theObviousEquivalence} shows that the condition \eqref{eq: theThirdTrivialConditon} is equivalent to the right-hand side of \eqref{eq: theCharacterizationOfWeightIsOne}. This concludes the proof.
\end{proof}
Before proving Lemma \ref{lem: theRaney-typeLemma}, we will show how to convert it into the framework of Raney's lemma: For $(m,n)\in \mathcal{M}_{d,k}\times \mathcal{N}_{d,k}$ we associate a vector $y=y(m,n)\in\Z^{d+1}$, endowed with cyclic indexing and defined by
    \begin{equation}\label{eq: theVectorAssociatedToMN}
        y(m,n)_j= 1 + \Sigma(n,\Sigma(m,j-1))-\Sigma(n,\Sigma(m,j)),\quad j\in\Z.
    \end{equation}
We stress that the right-hand side of \eqref{eq: theVectorAssociatedToMN} is $(d+1)$-periodic in $j$ (the definition is therefore compatible with the cyclic indexing of $y$); this is easily seen from \eqref{eq: theAlgebraPropertyOfSum} if one exploits that $T^{d+1}m=m$, $T^kn=n$, $\Sigma(m,d+1)=k$, and $\Sigma(n,k)=d$. This vector has various good properties:
\begin{lemma}\label{lem: theGoodPropertiesOfTheRaneySequence}
    For $(m,n)\in \mathcal{M}_{d,k}\times \mathcal{N}_{d,k}$ let $y=y(m,n)\in \Z^{d+1}$ be as in \eqref{eq: theVectorAssociatedToMN}. Then:
    \begin{enumerate}[label=\roman*)]
    \item The vector $y$ sums to one, $y_1+\dots+ y_{d+1}=1$.
    \item The weight $S(m,n)$ is one (and otherwise zero) if, and only if, the associated vector has positive partial sums, $y_1+\dots+y_j>0$ for all $j\in\{1,\dots,d+1\}$.
    \item A cyclic shift of $y$ is related to cyclic shifts of $(m,n)$ through the formula 
    \begin{equation}\label{eq: theSingleShiftRelation}
        T\big(y(m,n)\big)=y(Tm,T^{m_1}n).
    \end{equation}
\end{enumerate}
\end{lemma}
\begin{proof}
Since $\Sigma(n,\Sigma(m,0))=0$, we have
\begin{equation*}%\label{eq: partialSumsOfTheYSequence}
    y_1+\dots+y_j = j - \Sigma(n,\Sigma(m,j)),\qquad j\geq 1.
\end{equation*}
Thus $i)$ follows because $\Sigma(n,\Sigma(m,d+1))=\Sigma(n,k)=d$, while $ii)$ follows from Lemma \ref{lem: aSimplerCharacterizationOfWhenTheCompositionOfShiftsIsOne}. 
 We get $iii)$ by using \eqref{eq: theAlgebraPropertyOfSum}:
\begin{align*}
    (Ty)_j =&\,1 + \Sigma\Big(n,\Sigma(m,j)\Big)-\Sigma\Big(n,\Sigma(m,j+1)\Big)\\
    =&\,1 + \Sigma\Big(n,m_1 + \Sigma(Tm_1,j-1)\Big)-\Sigma\Big(n,m_1+\Sigma(Tm,j)\Big)\\
    =&\,1 +\Sigma(n,m_1) +\Sigma\Big(T^{m_1}n,\Sigma(Tm_1,j-1)\Big)-\Sigma(n,m_1) -\Sigma\Big(T^{m_1}n,\Sigma(Tm,j)\Big)\\
    =&\, y(Tm,T^{m_1}n)_j.\qedhere
\end{align*}
\end{proof}
We are finally sufficiently equipped to prove the Raney-type lemma.
\begin{proof}[Proof of Lemma \ref{lem: theRaney-typeLemma}]
On $\mathcal{M}_{d,k}\times \mathcal{N}_{d,k}$ we define the permutation $\sigma(m,n)\coloneqq (T m,T^{m_1}n)$, and so
\begin{equation*}%\label{eq: identityWithSigma}
   \sigma^j(m,T^\ell n)= (T^jm,T^{\ell+s(j)} n),\qquad j,\ell\in\Z,
\end{equation*}
where $s(j)\coloneqq\Sigma(m,j)$. Because $\{Tn, \dots ,T^kn\}=\{T^{1+s(j)}n, \dots ,T^{k+s(j)}n\}$, it follows that 
\begin{equation}\label{eq: rewritingTheWeightSumInTermsOfThePerturbation}
   \sum_{\ell=1}^kS(T^jm,T^\ell n)= \sum_{\ell=1}^kS(T^jm,T^{\ell+s(j)} n)=\sum_{\ell=1}^kS\big(\sigma^j(m,T^{\ell} n)\big),
\end{equation}
for any $j\in\Z$.
Next, in light of Raney's lemma (Lemma \ref{lem: RaneysLemma}) and $i)$ from Lemma \ref{lem: theGoodPropertiesOfTheRaneySequence}, we infer for $y=y(m,T^\ell n)$ that there is precisely one $r\in\{1,\dots,d+1\}$ such that $T^ry$ has positive partial sums. And by $ii)$ and $iii)$ in Lemma \ref{lem: theGoodPropertiesOfTheRaneySequence} it follows that 
\begin{equation*}
     S\big(\sigma^j(m,T^\ell n)\big)=\delta_{jr}\qquad \implies \qquad \sum_{j=1}^{d+1}S\big(\sigma^j(m,T^\ell n)\big)=1.
\end{equation*}
The lemma follows by summing \eqref{eq: rewritingTheWeightSumInTermsOfThePerturbation} over $j\in\{1,\dots,d+1\}$ and interchanging the order of the sums.
\end{proof}

\subsection{Properties of the zero-dispersion limit}\label{sec: propertiesOfTheZeroDispersionLimit}
In this section we prove Theorem \ref{thm: propertiesOfZD}; the proof will rely solely on properties of the mapping $AS[\cdot]$ introduced earlier, but which we define again here in equivalent, but more convenient, notation: For $v\in L^1(\T,\R)$, $t\in\R$, and a.e.~$x\in\T$, we set
\begin{equation}\label{eq: newDefinitionsOfAS}
    AS[v](t,x)\coloneqq \int_{0}^\infty f_v(x-2ty,y)\, dy,\qquad f_v(x,y)\coloneqq \mathbbm{1}_{v(x)>y>0} -\mathbbm{1}_{v(x)<-y<0}.
\end{equation}
\begin{proof}[Proof of Theorem \ref{thm: propertiesOfZD}]
We prove the maximum principle first as it will be useful later.

\noindent\textit{Maximum principles:} For $v\in L^\infty(\T)$ it is obvious from definition \eqref{eq: newDefinitionsOfAS} that
\begin{equation*}
    -\mathbbm{1}_{\mathrm{ess}\inf v<-y<0}\leq f_v(x,y)\leq \mathbbm{1}_{\mathrm{ess}\sup v>y>0},
\end{equation*}
and so the maximum principle follows. In particular one gets $\|AS[v](t)\|_{L^\infty}\leq \|v\|_{L^\infty}$.

\noindent\textit{Joint time- and $L^p$-continuity:} Let $p\in[1,\infty)$, $v,w\in L^p(\T,\R)$,  $t\in \R$, and let $[h]^p\coloneqq h|h|^{p-1}$ denote the signed power. Using a layer cake representation, one finds that 
\begin{align*}%\label{eq: differenceOfASIsIntegralInF}
       \int_0^\infty py^{p-1}|(f_v-f_w)(x,y)|\, dy =&\, \big|[v]^{p}-[w]^{p}\big|(x),
\end{align*}
Using this identity, the bound $\|f_v-f_w\|_{L^\infty}\leq 2$, and Lemma \ref{lem: weightedInequality}, we have
 \begin{equation}\label{eq: longComputation}
 \begin{split}
        \|\big(AS[v]-AS[w]\big)(t)\|_{L^p(\T)}^p \leq &\,\|\,\|(f_v-f_w)(x-2ty,y)\|_{L^1_y(\R_+)}\|_{L^p_x(\T)}^p \\
        \leq&\, 2^{p-1}\int_\T\int_{0}^\infty  py^{p-1}|(f_v-f_w)(x-2ty,y)|\, dy\,dx\\
        =&\, 2^{p-1}\int_\T\int_{0}^\infty py^{p-1}|(f_v-f_w)(x,y)|\, dy\,dx\\
        =&\, 2^{p-1}\int_{\T}\big|[v]^{p}-[w]^{p}\big|\, dx\\
        \leq &\, 2^{p-1}\int_{\T}p\big(|v|+ |w|\big)^{p-1}|v-w|\, dx\\
        \leq&\, p2^{p-1}\big(\|v\|_{L^p(\T)}+\|w\|_{L^p(\T)}\big)^{p-1}\|v-w\|_{L^p(\T)}.     
 \end{split}
    \end{equation}
That is, $v\mapsto AS[v](t)$ is locally uniformly $\frac{1}{p}$-Hölder continuous on $L^p(\T)$ with a $t$-independent coefficient. Let $v_n\to v\in L^p(\T)$ and $t_n\to t\in\R$. Using \eqref{eq: longComputation} and the triangle inequality we get 
\begin{align*}
    \limsup_{n\to\infty} \|AS[v_n](t_n)-AS[v](t)\|_{L^p(\T)}\leq&\,   \limsup_{n\to\infty}\|AS[v](t_n)-AS[v](t)\|_{L^p(\T)}\\
    \leq&\,\lim_{\varepsilon\to 0}\limsup_{n\to\infty}\|AS[v^\varepsilon](t_n)-AS[v^{\varepsilon}](t)\|_{L^p(\T)},
\end{align*}
where $v^\varepsilon\in C^1(\T)$ is such that $\|v-v^{\varepsilon}\|_{L^p(\T)}<\varepsilon$. Thus, it remains to show that  $AS[\varphi]\in C(\R, L^p(\T))$ for any $\varphi\in C^1(\T)$: For this, note first that $|f_\varphi(x,y)-f_{\varphi}(z,y)| =0$, if either  $|y|>\|\varphi\|_{L^\infty}$ or if $|y-\varphi(x)|>|x-z|\|\varphi'\|_{L^\infty}$. Using both of these two observations, and setting $z=x+\delta y$, it holds that
\begin{align*}%\label{eq: boundForTimeContinuityOfAS}
    |f_\varphi(x,y)-f_{\varphi}(x+\delta y,y)|&\, =0,\text{ if } |y-\varphi(x)|>|\delta|\|\varphi\|_{L^\infty}\|\varphi'\|_{L^\infty},\\
    \implies \int_{0}^\infty|f_\varphi(x,y)-f_{\varphi}(x+\delta y,y)|&\,\,dy\leq |\delta|\|\varphi\|_{L^\infty}\|\varphi'\|_{L^\infty}.
\end{align*}
Picking $s,t\in \R$, we get the desired time-continuity 
\begin{align*}
    \|AS[\varphi](t)-AS[\varphi](s)\|_{L^p(\T)}^p\leq &\,(2\|\varphi\|_{L^\infty})^{p-1}  \|AS[\varphi](t)-AS[\varphi](s)\|_{L^1(\T)} \\
        \leq&\, (2\|\varphi\|_{L^\infty})^{p-1} \int_{\T}\int_{0}^\infty|f_\varphi(x-2ty,y)-f_\varphi(x-2sy,y)|\,dy\, dx\\
       =&\, (2\|\varphi\|_{L^\infty})^{p-1} \int_{\T}\int_{0}^\infty|f_\varphi(x,y)-f_\varphi(x+2(t-s)y,y)|\,dy\, dx\\
      \leq&\, 2\pi(2\|\varphi\|_{L^\infty})^{p} \|\varphi'\|_{L^\infty}|t-s|.
\end{align*}

\noindent\textit{Norm control and $L^1$ contractions:}  The control of the $L^\infty$ norm was demonstrated at the beginning. The $p\in[1,\infty)$ case is dealt with using Lemma \ref{lem: weightedInequality}:
\begin{equation}\label{eq: theLpControlComputation}
    \|AS[v](t)\|_{L^p(\T)}\leq \int_{\T}\int_0^\infty py^{p-1}|f_v(x-2ty,y)|\, dy\,dx=\int_{\T}\int_0^\infty py^{p-1}|f_v(x,y)|\, dy\,dx=\|v\|_{L^p(\T)}.
\end{equation}
The $L^1$-contractions are obtained from \eqref{eq: longComputation} when setting $p=1$.

\noindent\textit{Oleinik estimate:} It will here be convenient to use notation more akin to the original one for $AS$, namely $AS[v](t,x)=\int_{\R}\chi_v(x-2ty,y)\,dy$ where $\chi_v(x,y)\coloneqq \mathbbm{1}_{v(x)>y>0} -\mathbbm{1}_{v(x)<y<0}$. This is because of the pointwise bound $\chi_v(x,y+h)\leq \chi_v(x,y) + \mathbbm{1}_{0>y>-h}$, which is valid for $h>0$, $x\in \T$, and $y\in\R$. Suppose $t>0$, then
\begin{align*}
    AS[v](t,x+h)=&\, \int_\R \chi_v(x+h-2ty,y)  \, dy=\int_\R \chi_v\big(x -2ty,y+\tfrac{h}{2t}\big)  \, dy\\ \leq &\,     \int_{\R}\chi_v(x-2ty,y) +\mathbbm{1}_{0>y>-h/2t}\,dy
     = AS[v](t,x) + \frac{h}{2t}.
\end{align*}
The remaining cases, $h,t<0$, $h<0<t$, and $t<0<h$, can be dealt with similarly; some of the inequalities will flip, but in all cases the resulting estimate can be written as in Theorem \ref{thm: propertiesOfZD}. 

\noindent\textit{Strong convergence before breaking:} We need only prove that $\|AS[u_0](t)\|_{L^2(\T)}=\|u_0\|_{L^2(\T)}$ for $t\in[T_-,T_+]$ since weak convergence plus convergence of norms implies strong convergence in $L^2(\T)$. For fixed $t\in (T_-,T_+)$ and $x\in \T$, note that there is a unique solution $y_*$ to the equation $0=h(y)\coloneqq y-u_0(x-2ty)$ since $h'(y)>1$ with $h(\pm\infty)=\infty$. As a consequence, we get that $y\mapsto f_v(x-2ty,y)$ is of the form $y\mapsto \mathrm{sgn}(y_*)\mathbbm{1}_{0<y<|y_*|}$. This further implies that $|\int_{0}^\infty f_v(x-2ty,y)\, dy| = \int_{0}^\infty |f_v(x-2ty,y)|\, dy$ and that the bound of Lemma \ref{lem: weightedInequality} is in fact an equality when applied to $y\mapsto f(x-2ty,y)$. Thus, the inequality in \eqref{eq: theLpControlComputation} becomes an equality for all $p\in[1,\infty)$; in particular for $p=2$. By the $L^2$-continuity of $AS[\cdot]$, this extends up to the endpoints $t\in\{T_-,T_+\}$ as well. 

\noindent\textit{Entropy solution from Trotter's formula:} The mapping $(t,u_0)\mapsto ZD[u_0](t)$ coincides with the transport-collapse map $(u_0,t)\mapsto T(t)u_0$ in \cite{Brenier} when the flux-function is defined $B(u)=u^2$. Strictly speaking, it is assumed in \cite{Brenier} that $B$ is globally Lipschitz, $x\in \R$, and $u_0\in L^1(\R)$. However if $u_0$ is bounded, then by maximum principles and finite speed of propagation (both of the map $T$ and of the entropy solution $u_{\mathrm{ent}}$), the result in \cite{Brenier} carries over to our setting. In particular, we get from \cite[Theorem 1]{Brenier} that
\begin{equation}
    u_{\mathrm{ent}}(t) =\lim_{n\to\infty}T(\tfrac{t}{n})^{\circ n}u_0,\quad \text{in }L^1(\T),
\end{equation}
 for all $t\geq 0$, and $u_0\in L^\infty(\T)$. The extension to $u_0\in L^1(\T)$ follows by density when using the $L^1$-contractions.
\end{proof}
\section*{Acknowledgments}
The author is grateful to the anonymous reviewers for valuable suggestions that improved the manuscript and would like to thank P. Gérard, J-C. Saut, P. Miller, X. Chen, and L. Gassot for helpful discussions during the writing of this paper.
The author is supported by the French Agence Nationale de la Recherche (ANR) under grant number ANR-23-CE40-0015 (ISAAC).

\appendix
\section{Auxiliary lemmas}
The following lemma is a justification of the identity \eqref{eq: theGeneralizedMillerXuFormula}.
\begin{lemma}\label{lem: theObviousLemma}
    Let $u_0\in C^1(\T,\R)$, and let $(t,x)$ denote points in $\in\R\times \T$. Define the multivalued solution of the Burgers equation \eqref{eq: burgersEquation} and the set of caustics (the envelopes of the characteristics) by
    \begin{align*}
        u_B(t,x)\coloneqq&\, \{y\colon y=u_0(x-2ty)\}, & \mathcal{C}\coloneqq&\, \big\{(t,x)\colon \exists z\in\R, \,g_t(z)=x,\, g'_t(z)=0\big\},
    \end{align*}
where $g_t\in C^1(\R,\T)$ is defined $g_t(z)=z+2tu_0(z)$.
Then $\mathcal{C}$ is a Lebesgue null set. Moreover, if $(t,x)\notin \mathcal{C}$, then $u_B(t,x)$ contains an odd number of values $u_B^0(t,x)<\dots<u_B^{N}(t,x)$ satisfying 
\begin{equation*}
       \sum_{n=0}^{N}(-1)^nu^n_B(t,x) = \int_{\R}\chi(u_0(x-2ty),y)\, dy,\quad \text{where }\chi(u,y)\coloneqq \mathbbm{1}_{u>y>0} -\mathbbm{1}_{u<y<0},\quad u,y\in\R.
\end{equation*}
\end{lemma}
\begin{proof}
    Since $u_0,u_0'$ are continuous, $\mathcal{C}$ is a closed set and therefore measurable. Thus, $\mathcal{C}$ is a null set in $\R\times \T$ because each fiber $\mathcal{C}_t=\{x\colon (t,x)\in \mathcal{C}\}$, which is precisely the set of critical values of $g_t$, is a null set in $\T$ by Sard's theorem.
   
   Next, fix $(t,x)\notin \mathcal{C}$, then any root $y_*$ of the function $h\in C^1(\R,\R)$, defined
$       h(y)\coloneqq y-u_0(x-2ty)$,
   must be simple, because if $h(y_*)=h'(y_*)=0$ then $(t,x)\in \mathcal{C}$ with corresponding $z=x-2ty_*$. Furthermore, since $\lim_{y\to\pm\infty}h(y)=\pm\infty$, it follows that $h$ has an odd number of roots that we label $y_0<\dots<y_N$ with $N\in 2\N_0$; these roots make up the set $u_B(t,x)$. Using the trivial identity $y_n = \int_{\R}\mathbbm{1}_{y_n>y>0}-\mathbbm{1}_{y_n<y<0}\, dy$ we get that
   \begin{equation}\label{eq: theIntegralIdentityPrematurely}
         \sum_{n=0}^{N}(-1)^ny_n =   \int_{\R}\sum_{n=0}^{N}(-1)^n\mathbbm{1}_{y_n>y>0}\, dy - \int_{\R}\sum_{n=0}^{N}(-1)^n\mathbbm{1}_{y_n<y<0}\, dy.
   \end{equation}
Since $N$ is even, one sees that the integrand of the first integral is equal to one (and otherwise zero) precisely when both $y>0$ and an odd number of the roots $y_0,\dots,y_N$ exceed $y$; the latter condition coincides with the condition $0>h(y)$ (the roots are simple and $h(\infty)=\infty$) which is to say that $u_0(x-2ty)>y$. Thus, $
    \sum_{n=0}^{N}(-1)^n\mathbbm{1}_{y_n>y>0} =\mathbbm{1}_{u_0(x-2ty)>y>0}$.
Reformulating the second integrand in \eqref{eq: theIntegralIdentityPrematurely} similarly, one gets the desired identity of the lemma.
\end{proof}
\begin{remark}[Breakdown at the caustics]\label{rem: failureAtCaustics}
    If $u_0$ is analytic, then every root of $h$ from the previous proof has a well defined multiplicity. In this case, one may actually extend formula \eqref{eq: theGeneralizedMillerXuFormula} to the set of caustics if one counts the members of $u_B$ with multiplicity (inherited from the corresponding roots of $h$) so that, in particular, members with even multiplicity vanish. For $u_0\in C^1(\T)$, however, this multiplicity may not be well defined and, worse yet, there may be points $(t,x)\in \mathcal{C}$ for which $h$ has an infinite number of roots. For example, if $u_0(z)=1-z$ for $z\in [0,1]$ then, at the point $(t,x)=(\frac{1}{2},1)\in\mathcal{C}$, we find for any $y\in[0,1]$ that $h(y)= y-u_0(x-2ty)=y-u_0(1-y)=y-(1-(1-y))=0$. 
\end{remark}
The following small lemma is used in the proof of Theorem \ref{thm: propertiesOfZD}.
\begin{lemma}\label{lem: weightedInequality}
    Let $f\in L^1\cap L^\infty(\R_+,\R)$, and $p\in[1,\infty)$. Then
    \begin{equation}\label{eq: weightedInequality}
        \|f\|_{L^1(\R_+)}\leq p^{\frac{1}{p}}\|f\|_{L^\infty(\R_+)}^{1-\frac{1}{p}}\||\cdot|^{p-1}f\|_{L^1(\R_+)}^{\frac{1}{p}},
    \end{equation}
with equality whenever $f(x)=\alpha\mathbbm{1}_{0<x<R}$ for $R>0$ and $\alpha\in\R$.
\end{lemma}
\begin{proof}
   Set $F(x)\coloneqq \int_0^x |f(y)|\, dy\leq x\|f\|_{L^\infty(\R_+)}$. Then
    \begin{equation*}
 \|f\|_{L^1(\R_+)}^p=F(\infty)^p = p\int_0^\infty F^{p-1}(x)F'(x)\, dx\leq p\|f\|_{L^\infty(\R_+)}^{p-1}\int_0^\infty x^{p-1}|f(x)|\, dx.
    \end{equation*}
The last part of the lemma is trivial.
\end{proof}

\section*{Data Availability Statement}
Data sharing not applicable to this article as no datasets were generated or analyzed during the current study.
\section*{Conflict of Interest}
The author declares that there is no conflict of interest.
\bibliographystyle{abbrv}
\bibliography{Bibliography} 
\end{document}